\documentclass[]{birkjour}
\usepackage{amsmath}
\usepackage{amssymb}
\usepackage{mathabx}

\newtheorem{theorem}{Theorem}[section]
\newtheorem{proposition}[theorem]{Proposition}
\newtheorem{lemma}[theorem]{Lemma}
\newtheorem{corollary}[theorem]{Corollary}
\newtheorem{definition}[theorem]{Definition}

\theoremstyle{definition}

\numberwithin{equation}{section}
\newtheorem{example}[theorem]{Example}

\newcommand{\krein}{Kre{\u\i}n }
\newcommand{\ip}[2]{\left<#1,#2\right>}
\newcommand{\cF}{{\mathcal F}}
\newcommand{\cG}{{\mathcal G}}
\newcommand{\cH}{{\mathcal H}}
\newcommand{\cK}{{\mathcal K}}
\newcommand{\cL}{{\mathcal L}}
\newcommand{\cM}{{\mathcal M}}
\newcommand{\cN}{{\mathcal N}}
\newcommand{\cR}{{\mathcal R}}
\newcommand{\cS}{{\mathcal S}}
\newcommand{\cW}{{\mathcal W}}
\makeatletter
\newcommand*\bigcdot{\mathpalette\bigcdot@{1.8}}
\newcommand*\bigcdot@[2]{{{\hbox{\scalebox{#2}{$\m@th#1.$}}}}}
\makeatother
\newcommand{\pdot}{\,{\overset{\smash{\bigcdot}}{+}}\,}
\newcommand{\oplusdot}{\,{\overset{\smash{\bigcdot}}{\oplus}}\,}
\newcommand{\bigoplusdot}{\,{\overset{\smash{\bigcdot}}{\bigoplus}}\,}
\newcommand{\fdh}{\cH = \cH_+ \oplus \cH_-}
\newcommand{\lh}{\cL(\cH)}
\newcommand{\lhk}{\cL(\cH,\cK)}
\newcommand{\ran}{{\rm ran\,}}
\renewcommand{\ker}{{\rm ker\,}}
\newcommand{\linL}{{\lambda\in\Lambda}}
\newcommand{\bbN}{{\mathbb N}}
\newcommand{\myspan}{{\rm span\,}}  
\newcommand{\dinD}{{d\in D}}

\begin{document}

\title[Rearrangement Invariant Sums. II]{Rearrangement Invariant
  Orthogonal Sums in Kre{\u\i}n Spaces. II}

\author[M.\ A.\ Dritschel]{Michael A.\ Dritschel}

\address{School of Mathematics, Statistics, \& Physics, Newcastle
  University, Newcastle upon Tyne, NE1 7RU, United Kingdom}

 \email{michael.dritschel@newcastle.ac.uk}

 \author[A.\ Maestripieri]{Alejandra Maestripieri}

 \address{Instituto Argentino de Matem\'atica ``Alberto
   P. Calder\'on'' CONICET, Saavedra 15, Piso 3, (1083) Buenos Aires,
   Argentina}

 \email{amaestripieri@gmail.com}

 \author[J.\ Rovnyak]{James Rovnyak (corresponding author)}
 
 \address{Department of Mathematics, University of Virginia, P.O.\ Box
   400137, Charlottesville, Virginia, U.S.A.}

\email{rovnyak@virginia.edu}

\keywords{\krein spaces, regular spaces, pseudo-regular subspaces,
  operator ranges, quasi-pseudo-regular subspaces, orthogonal sums,
  direct sums, Moore-Smith convergence, Moore-Smith sums}

 \subjclass{47B50, 47A05, 46C20}

 \dedicatory{To our colleague and friend Henk de Snoo, on his
   $80^{\mathrm{th}}$ birthday.}

  \date{\today}

\begin{abstract}
  Part I of the paper considered infinite orthogonal sums of regular
  subspaces in a \krein space (that is, of subspaces which are
  themselves \krein spaces).  How precisely these sums should be
  defined and conditions for when such sums are themselves regular
  were determined.  These included, for example, a boundedness
  condition for the sum of the corresponding orthogonal projections.
  The same problem is addressed here for (quasi-)pseudo-regular
  subspaces.  Such subspaces are orthogonal direct sums of a regular
  spaces and isotropic, or neutral, subspaces.  Alternate
  characterizations of such subspaces are given, and infinite
  orthogonal sums are examined via unconditional, or Moore-Smith, sums
  of operator ranges.
\end{abstract}

\maketitle

\section{Introduction and background}\label{S:Introduction}

A fundamental result in the study of Hilbert space geometry is that
orthogonal sums of Hilbert spaces are themselves Hilbert spaces.  This
works whether the number of spaces involved is finite or infinite.
\krein spaces are indefinite analogues of Hilbert spaces, and so it is
natural to ask which properties of Hilbert spaces carry over to these.

Beginning in 1979, Brian~McEnnis made a detailed study of this topic
in a series of papers considering shift operators on \krein spaces
\cite{McEnnis1,McEnnis2,McEnnis3}.  Part~I of the current
series~\cite{Rovnyak2025} contains an elaboration and expansion of
McEnnis' work.  This, an earlier work of the third author
\cite{Rovnyak2021}, and McEnnis' papers were primarily concerned with
\textsl{regular spaces}, which are the natural analogue in \krein
spaces of closed subspaces of Hilbert spaces, in that they are
themselves \krein spaces and orthogonal sums are automatically direct.

Part~II will examine orthogonal sums of \textsl{pseudo-regular} and
\textsl{quasi-pseudo-regular} or \textsl{qpr} spaces (formally defined
in Definition~\ref{defpseudoreg}), which generalize regular spaces in
that they are the orthogonal direct sum of regular spaces and neutral
subspaces (where all inner products are $0$).  A natural question is
how the orthogonal sum of such subspaces should be defined, and what
conditions are needed to guarantee that orthogonal sums exist.  Here,
even for finitely many orthogonal spaces, the sum need not exist (as a
space of the same type), or it may not be direct, along with other
possible difficulties.

Theorem~\ref{maintheorem1} originates in Part I of this
paper~\cite{Rovnyak2025} and highlights the status regarding
orthogonal sums of regular spaces.  There, infinite sums of regular
spaces are well-defined in terms of rearrangement invariance of
infinite direct sums of vectors from these spaces, or equivalently, by
means of a boundedness condition of related projection operators.
While it is perhaps too much to ask that such definitive results hold
for pseudo-regular and quasi-pseudo-regular spaces, analogues are
nevertheless available --- see Theorems~\ref{ThC} and~\ref{qpr_sums},
as well as Corollary~\ref{cor:appl_to_pr_spaces}.  The first of these
theorems is a strengthening of Theorem 4.3 in \cite{Rovnyak2021}, the
main ideas of which were discovered by the first two authors and
communicated to the third.

Quasi-pseudo-regular (or qpr) spaces are new, being an outgrowth of
work of Aurelian Gheondea on pseudo-regular
spaces~\cite{Gheondea1984,Gheondea2022}.  The latter spaces are
assumed to be closed, while this is relaxed in the
quasi-pseudo-regular setting to the assumption that they are operator
ranges.  Generalizations of Gheondea's results form the bulk of
Section~2.  A deeper understanding of infinite sums of operator ranges
is required when considering orthogonal sums of quasi-pseudo-regular
spaces.  Section~3 takes this up, expanding on fundamental work of
Fillmore and Williams~\cite{FW1971}.  Section~4 deals with nets of
projections in \krein spaces, playing a key part in understanding
orthogonal sums of pseudo-regular spaces.  The last two sections put
all of this together in theorems on infinite orthogonal sums of
pseudo-regular and quasi-pseudo-regular spaces.

Now for a brief summary of the needed background in \krein space
theory following in the spirit of~\cite{DRFields}.

Recall that a \krein space is a complex inner product space
$(\cH, {\ip{\cdot}{\cdot}}_{\cH})$ which admits a \textbf{fundamental
  decomposition} $\fdh$ such that $\cH_+$ is a Hilbert space and
$\cH_-$ is the anti-space of a Hilbert space (that is, a Hilbert space
with the inner product replaced by its negative).  Every \krein space
$\cH$ has a unique \textbf{strong topology} which is induced by any
such decomposition, though in general the decomposition is not unique.
A subspace $\cM$ is called \textbf{regular} if it is closed and is
itself a \krein space in the inner product of~$\cH$.

In analogy with Hilbert spaces, the range of a selfadjoint projection
operator on $\cH$ is a regular subspace, and every regular subspace is
the range of a unique selfadjoint projection.  Here
\textsl{selfadjoint} means with respect to the indefinite inner
product.

\goodbreak

\begin{definition}
  \label{def:orth-sum_and_direct_sum}
  Let be $\cM_1$ and $\cM_2$ be $($not necessarily closed$)$ subspaces
  of a \krein space $\cH$.
  \begin{enumerate}
  \item $\cM_1$ and $\cM_2$ are \textbf{orthogonal} if for any
    $f\in \cM_1$, $g\in \cM_2$, $\ip{f}{g} = 0$.  $($Unlike in Hilbert
    spaces, this does \emph{not} necessarily imply that
    $\cM_1 \cap \cM_2 = \{0\}$.$)$
  \item The sum of orthogonal subspaces $\cM_1$ and $\cM_2$ is written
    as $\cM_1 \oplus \cM_2$ and called the \textbf{orthogonal sum}.
  \item If $\cM_1 \cap \cM_2 = \{0\}$, the sum of $\cM_1$ and $\cM_2$
    is written as $\cM_1 \pdot \cM_2$ and is called the \textbf{direct
      sum}.
  \item If $\cM_1$ and $\cM_2$ are orthogonal and
    $\cM_1 \cap \cM_2 = \{0\}$, the sum of $\cM_1$ and $\cM_2$ is
    written as $\cM_1 \oplusdot \cM_2$ and is called the
    \textbf{orthogonal direct sum}.
  \end{enumerate}
\end{definition}
Definitions (2)--(4) extend obviously to any finite number of spaces.
The notation is at odds with that used in Part I \cite{Rovnyak2025},
but since the goal here is to consider orthogonal sums which are not
necessarily direct, the above notation is better suited to the task.

Finite orthogonal sums of closed spaces in a Hilbert space are direct
and closed, and likewise finite orthogonal sums of regular subspaces
are direct and regular.

\goodbreak

\begin{theorem}\label{finiteorthosums}
  Let $\cM_1,\dots,\cM_n$ be pairwise orthogonal regular subspaces of
  a \krein space $\cH$.
  \begin{itemize}
  \item[$(1)$] The span $\cM$ of $\cM_1,\dots,\cM_n$ is regular, and
    $\cM$ is the orthogonal direct sum of the subspaces:
    \begin{equation*}
      \cM = \cM_1\, \oplusdot \cdots \oplusdot \cM_n.
    \end{equation*}    
    The selfadjoint projection $P$ onto $\cM$ is given by
    $P = E_1 + \cdots + E_n$, where $E_1,\dots,E_n$ are the
    selfadjoint projections onto $\cM_1,\dots,\cM_n$.
  \item[$(2)$] Let $\widetilde\cM = \cM_1\times\dots\times\cM_n$ be
    the Cartesian product of the subspaces viewed as \krein spaces in
    the inner product of~$\cH$.  Then $\widetilde \cM$ is a \krein
    space, and the mapping
    \begin{equation*}
      U\colon (f_1,\dots,f_n) \to f_1 + \cdots + f_n      
    \end{equation*}
    from $\widetilde\cM$ onto $\cM$ is a \krein space isomorphism.
  \end{itemize}
\end{theorem}

As indicated earlier, the generalization of Theorem
\ref{finiteorthosums} to infinite families is standard for Hilbert
spaces.  This is not the case for \krein spaces.  While the orthogonal
sum of regular subspaces is direct, it may fail to be regular, or it
may be regular but not every vector in the closed span is
representable as a sum of vectors from the individual
subspaces~\cite[Example~7.3]{McEnnis1},
\cite[Example~6.2]{Rovnyak2025}.  Likewise, the notion of an infinite
Cartesian product is problematic for \krein spaces.

In what follows, $\mathrm{span}_\lambda\, \cM_\lambda$ stands for the
vector space of finite sums of $f_\lambda\in \cM_\lambda$, while
$\bigvee_\lambda \cM_\lambda$ is the norm closure of this space.
Finding conditions for when an orthogonal sum exists for various
classes of subspaces is the central goal of this paper.

Recall from part~I that for a Banach space $X$ with norm $\|\cdot\|$
and a sequence $x_1,x_2,\dots$ in~$X$, the infinite series
$\sum_{j=1}^\infty x_j$ is termed \textbf{unconditionally convergent}
if every rearrangement is convergent.  An infinite series can also be
viewed as an indexed sum with no regard to order.  An arbitrary
indexed sum $\sum_{\linL} x_\lambda$ of vectors in $X$ is said to be
\textbf{Moore-Smith convergent} with \textbf{sum} or
\textbf{Moore-Smith sum} $x$ if for every $\varepsilon > 0$ there is a
finite subset $\cF$ of $\Lambda$ such that
$\big\| x - \sum_{\lambda\in \cG} x_\lambda\big\| < \varepsilon$ for
every finite subset $\cG$ of $\Lambda$ that contains~$\cF$.  Such a
vector $x$ is unique and denoted
$x = \sum^{\scalebox{0.4}{\text{MS}}}_{\linL} x_\lambda$.

Equivalently, suppose $\{x_\lambda\}_{\linL} \subset X$.  The
collection of all finite subsets $\mathbb F = \{\cF\}$ of $\Lambda$
ordered by inclusion is a directed set.  Define
$x_\cF := \sum_{\lambda\in \cF} x_\lambda$.  Then
$(x_\cF)_{\cF\in\mathcal F}$ is a net.  Convergence of this net to $x$
(that is, $x_\cF \to x$) means that given any $\varepsilon > 0$, there
is an $\cF\in\mathcal F$ such that for all $\cG\in\mathbb F$ with
$\cG \geq \cF$, $\|x - x_\cG\| < \varepsilon$.  Since the space $X$ is
Hausdorff, limits are unique.  The existence of a limit $x$ for
$(x_\cF)_{\cF}$ is equivalent to
$x = \sum^{\scalebox{0.4}{\text{MS}}}_{\linL} x_\lambda$.  The
definition of the directed set $\mathbb F$ ensures that the sum is
rearrangement invariant~\cite{Hildebrandt1940}.

The following definitions are minor variations on those used in part~I
of the paper~\cite[Definition~3.1]{Rovnyak2025} for regular subspaces.

\begin{definition}\label{defbigoplus}
  Let $\{\cM_\lambda\}_{\linL}$ be subspaces of a particular type
  $($for example, closed, or regular$)$ in a \krein space~$\cH$.  The
  subspace of elements of\/ $\bigvee_\lambda \cM_\lambda$ which can be
  expressed as Moore-Smith sums of elements of these subspaces is
  called the \textbf{Moore-Smith sum} and is denoted by
  \begin{equation*}
    \sum^{\scalebox{0.5}{\text{MS}}}_{\linL} \cM_\lambda.
  \end{equation*}
  So $x \in \sum^{\scalebox{0.4}{\text{MS}}}_{\linL} \cM_\lambda$, if
  and only if $x$ can be expressed as
  $\sum^{\scalebox{0.4}{\text{MS}}}_{\linL} x_\lambda$,
  $x_\lambda \in \cM_\lambda$, where the sum is Moore-Smith
  convergent.

  Any space $\cM$ where $\cM$ is of the same type as those in
  $\{\cM_\lambda\}_{\linL}$, and for which
  \begin{equation}
    \label{sum_def}
    \mathrm{span}_\lambda\, \cM_\lambda \subseteq \cM \subseteq
    \sum^{\scalebox{0.5}{\text{MS}}}_\lambda \cM_\lambda,
  \end{equation}
  is called a \textbf{sum} of these subspaces.  When the sum is
  unique, it is denoted as
  \begin{equation*}
    \sum_\linL \cM_\lambda.
  \end{equation*}
  
  If the subspaces are pairwise orthogonal, call $\cM$ an
  \textbf{orthogonal sum}.  When unique, this is denoted by
  \begin{equation*}
    \bigoplus_\linL \cM_\lambda.
  \end{equation*}
  
  If additionally, the sum is direct $($that is, a Moore-Smith
  convergent sum of terms in the spaces $\cM_\lambda$, $\linL$, is
  zero only if each term is zero$)$ call $\cM$ an \textbf{orthogonal
    direct sum} of the subspaces, and write
  \begin{equation*}
    \overset{\smash{\bigcdot}}{\bigoplus_{\linL}} \cM_\lambda
  \end{equation*}
  when the sum is unique.
\end{definition}

\goodbreak

It is worth emphasizing that all the sums are assumed to consist of
elements which are Moore-Smith sums.  Clearly, the span of any finite
set of subspaces is the Moore-Smith sum of those spaces, and for any
collection $\{\cM_\lambda\}_{\linL}$ of subspaces,
$\mathrm{span}_\linL \cM_\lambda \subseteq
\sum^{\scalebox{0.4}{\text{MS}}}_{\linL} \cM_\lambda$ densely.

As will be demonstrated in Example~\ref{exmpl:minimal_or_not}, when
they exist, sums, orthogonal sums, and orthogonal direct sums need not
be unique, even if the spaces being summed happen to be closed.  This
occurs for example if the spaces considered are viewed as operator
ranges, and so the sum is also an operator range (and not necessarily
closed).

There will be circumstances where uniqueness does hold.  If the spaces
in the class being summed are assumed to be closed (for example,
regular spaces, and the sum is then regular), the next lemma shows
that the sum will be unique.

\begin{lemma}\label{sum_closures}
  If $\{\cM_\lambda\}_{\linL}$ are subspaces of a \krein space and a
  sum $\cM$ exists, then its closure $\widebar\cM$ is unique and
  coincides with $\bigvee_\linL\cM_\lambda$.  In particular, if
  $\{\cM_\lambda\}_{\linL}$ is a family from a class of closed
  subspaces and a sum of the family $\{\cM_\lambda\}_{\linL}$ exists
  $($and so is also closed$)$, then it is unique and
  \begin{equation*}
    \sum_\linL \cM_\lambda = \sum^{\scalebox{0.5}{\text{MS}}}_{\linL}
    \cM_\lambda = \bigvee_\linL \cM_\lambda. 
  \end{equation*}
\end{lemma}

\begin{proof}
  By \eqref{sum_def}, if $\cM$ is a sum of the family
  $\{\cM_\lambda\}_{\linL}$,
  \begin{equation}\label{sum_def_2}
    \bigvee_\linL \cM_\lambda = \overline{\myspan_{\linL}
      \cM_\lambda} \subseteq \widebar \cM \subseteq
    \overline{ \sum^{\scalebox{0.5}{\text{MS}}}_\linL \cM_\lambda}
    \subseteq \bigvee_\linL \cM_\lambda.
  \end{equation}
  Therefore $\widebar \cM = \bigvee_\linL \cM_\lambda$.  If the
  subspaces $\{\cM_\lambda\}_{\linL}$ are closed and a sum $\cM$
  exists, then by definition it has to be closed, and by
  \eqref{sum_def_2}, $\cM$ is unique and
  \begin{equation*}
    \cM = \sum_\linL \cM_\lambda =
    \sum^{\scalebox{0.5}{\text{MS}}}_{\linL} \cM_\lambda 
    =\bigvee_\linL \cM_\lambda.
  \end{equation*}
\end{proof}

Here is the theorem which acts as a model for the results on infinite
orthogonal pseudo-regular and quasi-pseudo-regular spaces.

\begin{theorem}[\!\!\cite{Rovnyak2025}]
  \label{maintheorem1}
  Let $\{\cM_\lambda\}_{\linL}$ be pairwise orthogonal regular
  subspaces of a \krein space~$\cH$, and let $\{E_\lambda\}_{\linL}$
  be the selfadjoint projections in $\lh$ such that
  $\cM_\lambda = \ran E_\lambda$, $\linL$.  The following are
  equivalent:
  \begin{enumerate}
  \item[(1)] The orthogonal direct sum of the subspaces exists
    uniquely and
    \begin{equation*}
      \bigoplusdot_\linL \cM_\lambda = \bigvee_\linL \cM_\lambda.
    \end{equation*}
  
    \smallskip
  \item[(2)] For some and hence any norm $\|\cdot\|$ for~$\cH$, there
    is a constant $C > 0$ such that
    $\|\sum_{\lambda\in \cF} E_\lambda \| \le C$ for every finite
    subset $\cF$ of $\Lambda$.

    \smallskip
  \item[(3)] There is an operator $P\in\lh$ such that
    $P=\sum^{\scalebox{0.4}{\text{MS}}}_{\linL} E_\lambda$ in the
    sense that
    \begin{equation*}
      Pf = \sum^{\scalebox{0.5}{\text{MS}}}_{\linL} E_\lambda f,
      \qquad f\in\cH.
    \end{equation*}
  \item[(4)] There is a family $\{F_\lambda\}_\linL$ of selfadjoint
    projections on a Hilbert space~$\cK$, and an invertible operator
    $X\in\lhk$, such that $E_\lambda = X^{-1} F_\lambda X$ for all
    $\linL$.

    \smallskip
  \item[(5)] The subspace $\cM=\bigvee_\linL\cM_\lambda$ is regular,
    and $\sum_\linL \| E_\lambda f\|^2<\infty$ for all $f\in\cM$ for
    some and hence any associated norm for~$\cH$.
  \end{enumerate}
  In this case, the operator $P$ in $(3)$ is the selfadjoint
  projection with range~$\cM$.
\end{theorem}

\begin{corollary}[\!\!\cite{Rovnyak2021}]
  \label{maincorollary}
  When the equivalent conditions in Theorem $\ref{maintheorem1}$ hold,
  then for some and hence any associated norm for~$\cH$, the formula
  \begin{equation*}
    f = \sum^{\scalebox{0.5}{\text{MS}}}_{\linL} f_\lambda
  \end{equation*}
  determines a one-to-one correspondence
  $f \leftrightarrow \{f_\lambda\}_\linL$ between the elements $f$ of
  $\bigvee_\linL \cM_\lambda$ and the set of all families
  $\{f_\lambda\}_\linL$, $f_\lambda\in\cM_\lambda$, $\linL$, such that
  $\sum_\linL \|f_\lambda\|^2 < \infty$.  If $f$ and
  $\{f_\lambda\}_\linL$ are connected in this way,
  \begin{equation*}
    C_1 \|f\|^2 \le \sum_\linL \|f_\lambda\|^2 \le C_2 \| f \|^2 
  \end{equation*}
  for some constants $C_1 > 0$ and $C_2 > 0$.
\end{corollary}

For background to Theorem \ref{maintheorem1} and previous results, see
the citations in \cite{Rovnyak2025} and the remarks following Theorem
\ref{ThB} below.

\goodbreak

\section{Pseudo-regular and quasi-pseudo-regular subspaces}
\label{S:pseudoregular}

There are various ways to weaken the condition of regularity,
including those defining pseudo-regularity and
quasi-pseudo-regularity.  A compelling reason for the interest in
pseudo-regular subspaces is that every closed subspace of a Pontryagin
space (that is, a \krein space $\cH$ where either $\cH_+$ or $\cH_-$
is finite dimensional) is pseudo-regular (see Proposition 3.2.9
in~\cite{Gheondea2022} or
Corollary~\ref{cor:closed_subsp_in_Pontr_sp_ps-reg} below), while in
any \krein space, pseudo-regular spaces are quasi-pseudo-regular and
as will be seen, the latter have better properties under orthogonal
direct sums.  Several equivalent ways of defining pseudo-regular
subspaces exist (see Gheondea \cite[\S3.2]{Gheondea2022}).  The one
highlighting its connection with regularity is chosen here.

\goodbreak
\begin{definition}\label{defpseudoreg}
  A subspace $\cS$ of a \krein space $\cH$ is
  \textbf{quasi-pseudo-regular} $($\textbf{qpr} for short$)$;
  respectively, \textbf{pseudo-regular}, if
  \begin{enumerate}
  \item[(1)] $\cS$ is an operator range; respectively, closed; and
    
    \smallskip
  \item[(2)] $\cS = \cR\oplus\cN$, where $\cR$ is regular and $\cN$ is
    a neutral subspace.
  \end{enumerate}
\end{definition}

Lemma \ref{lem:M_closed_then_R_plus_M_closed} shows that the sum in
part (2) of Definition \ref{defpseudoreg} is in fact direct.

An \textbf{operator range} is the range of a bounded Hilbert space
operator.  Since \krein spaces can be identified with not necessarily
unique, but all equivalent, Hilbert spaces, the range of any bounded
\krein space operator is an operator range in this sense.  Using the
polar decomposition, one sees that any operator range is the range of
a positive (that is, positive semi-definite) Hilbert space operator.

As noted, closed subspaces of Pontryagin spaces are pseudo-regular,
though this does not hold for general \krein spaces.  Moreover,
orthogonal direct sums of pseudo-regular subspaces may fail to be
pseudo-regular, making the consideration of direct sums of such spaces
more complex.

\begin{example}\label{oct29a2024}
  \label{exmpl:sum_of_pr_not_qpr}
  Let $\cH$ be a \krein space with fundamental decomposition $\fdh$
  such that both $\cH_+$ and $\cH_-$ are infinite dimensional.  Then
  $\cH$ contains a closed neutral subspace $\cH_0$ with
  $\dim\cH_0=\infty$.  Choose closed subspaces $\cN$ and $\cN'$ of
  $\cH_0$ in such a way that $\cN\cap\cN'=\{0\}$ while the angle
  between them is zero (Halmos \cite[\S15]{Halmos1951}).  Since
  $\cH_0$ is neutral, $\cN\perp\cN'$.  Then $\cN\oplus\cN'$ is not
  closed, and so while $\cN$ and $\cN'$ are pseudo-regular,
  $\cN\oplus\cN'$ is not.
\end{example}

By contrast, as will be shown in Proposition~\ref{oct28b2024},
quasi-pseudo-regular subspaces are closed under the formation of
orthogonal sums.

As the next lemma shows, the difference between a qpr and
pseudo-regular space is that for the latter, the neutral summand is
necessarily closed.  Pseudo-regular subspaces are quasi-pseudo-regular
since any closed space is the range of a bounded Hilbert space
projection.  Regular subspaces are pseudo-regular spaces with zero
neutral component.

In any qpr space, $\cN$ in part (2) of Definition \ref{defpseudoreg}
is uniquely determined.

\begin{lemma}\label{lem:Lemma0}
  Let $\cH$ be a \krein space and assume that $\cS=\cR\oplus\cN$,
  where $\cR$ is regular and $\cN$ is neutral.  Then $\cN=\cS^0$,
  where $\cS^0 = \cS\cap\cS^\perp$ is the isotropic part of~$\cS$.
  Furthermore, $\overline{\cS^0} = {\overline{\cS}}^0$.  Also, $\cN$
  is closed, respectively, an operator range, if and only if $\cS$ is
  pseudo-regular, respectively, qpr.
\end{lemma}

\goodbreak

\begin{proof}
  Since $\cN$ is neutral, $\cN\perp\cN$, and so $\cN\perp\cS$ because
  $\cN\perp\cR$; hence $\cN\subseteq\cS\cap\cS^\perp$.  Clearly,
  $\cR^\perp\cap\cS\supseteq\cN$.  Conversely, suppose
  $x\in\cR^\perp\cap\cS$, and write $x=u+v$, $u\in\cR$, $v\in\cN$.
  For any $u'\in\cR$,
  \begin{equation*}
    0 = \ip{x}{u'} = \ip{u+v}{u'} = \ip{u}{u'},
  \end{equation*}
  so $u \in \cR^\bot$ and by the nondegeneracy of~$\cR$, $u=0$ and
  hence $x=v\in\cN$.  Thus $\cR^\perp\cap\cS=\cN$.  So
  $\cS^\perp\cap\cS \subseteq \cR^\perp\cap\cS=\cN$, and hence
  $\cS^\perp\cap\cS=\cN$.

  Since $\cR \subseteq \cS$ and $\cN \subset \cR^\bot$ by regularity
  of $\cR$, it follows that
  \begin{equation*}
    \cR \cap \cS^\perp = \cR \cap \cS \cap \cS^\perp =\cR \cap \cN
    \subset \cR \cap \cR^\bot = \{0\},
  \end{equation*}
  and so
  \begin{equation*} {\overline{\cS}}^0 = \overline{(\cR \oplus \cN)}
    \cap \cS^\perp = (\cR + \overline{\cN}) \cap \cS^\perp =
    \overline{\cN} \cap \cS^\perp = \overline{\cN} = \overline{\cS^0}.
  \end{equation*}

  For the last statement, if $\cS$ is pseudo-regular, and so closed,
  then $\cN$ is closed as it is the intersection of closed spaces.
  Likewise, if $\cS$ is an operator range, $\cS^\bot$ is also an
  operator range as it is closed, and so $\cN$ being the intersection
  of operator ranges is then an operator range
  \cite[Cor.~2,~p.~260]{FW1971}.
\end{proof}

Next, a couple of simple yet useful lemmas.

\begin{lemma}
  \label{lem:M_closed_then_R_plus_M_closed}
  If $\cR$ is a regular subspace of a \krein space $\cH$ and
  $\cM \subseteq \cR^\bot$ is closed $($respectively, an operator
  range$)$, then $\cR \cap \cM = \{0\}$ and
  $\cR \oplus \cM = \cR\, \dot{\oplus} \,\cM $ is closed
  $($respectively, an operator range$)$.
\end{lemma}

\begin{proof}
  By regularity, $\{0\} = \cR \cap \cR^\bot \supseteq \cR \cap \cM$.
  Choose a fundamental symmetry $J$ such that $J\cR = \cR$ and write
  $|\cH|$ for the associated Hilbert space.  Then the Hilbert space
  orthogonal complement of $\cR$ in $|\cH|$ equals $\cR^\bot$, and so
  $\cR$ and $\cM$ are Hilbert space orthogonal.  If $\cM$ is assumed
  to be closed, it then follows that $\cR \oplus \cM$ is closed.  On
  the other hand, if $\cM$ is an operator range, since $\cR$ is the
  range of a Hilbert space projection it follows by Theorem~2.2 of
  \cite{FW1971} that the sum is an operator range.
\end{proof}

\begin{lemma}
  \label{lem:S+Sperp_closed}
  If $\cS$ is a qpr subspace of a \krein space $\cH$, then
  $\cS + \cS^\bot$ is closed.
\end{lemma}

\begin{proof}
  Write $\cS = \cR\oplus\cN$, where $\cR$ is regular and
  $\cN = \cS^0 = \cS \cap \cS^\bot$ is neutral.  Since
  $\cS^\bot \subseteq \cR^\bot$ and $\cS^\bot$ is closed, it follows
  by Lemma~\ref{lem:M_closed_then_R_plus_M_closed} that
  \begin{equation*}
    \cS + \cS^\bot = \cR \oplus \cS^\bot
  \end{equation*}
  is closed.
\end{proof}

If $\cR$ and $\cS$ are orthogonal subspaces of a \krein space $\cH$, then the sum
$\cR + \cS^0$ is automatically orthogonal since
$\cS^0 \subseteq \cS$, though it may not be direct, even if $\cS$
is qpr.  However, if $\cS$ is qpr and the sum \textit{is} direct, the
proof of the next proposition shows that $\cR$ is a \krein space.  A
subtle point is that this may not be enough to ensure that $\cR$ is
regular, since while it is closed in the induced inner product, it may
not be closed as a subspace of $\cH$.  As the next example indicates,
something more is needed to guarantee regularity.

\begin{example}\label{exmpl:R_not_regular_in_S}
  Let $\cH$ be an infinite dimensional Hilbert space, $x,y \in \cH$,
  $\|x\| = \|y\| = 1$ and $\ip{x}{y} = 0$.  Set
  $\cH_0 = \cH \ominus \,\mathrm{span}\,\{x,y\}$,
  $\cH_1 = \cH \ominus \,\mathrm{span}\,\{y\}$.  Define a fundamental
  symmetry on $\cH$ by $J|\cH_0 = 1_{\cH_0}$ and $Jx=y$ (so $Jy = x$).
  Then $\cN := \mathrm{span}\,\{x\}$ is a neutral subspace both \krein
  and Hilbert space orthogonal to $\cH_0$.  It is a standard result
  that there exists an unbounded everywhere defined linear functional
  $\varphi$ on $\cH_1$ with $\varphi(x) = 1$.  Furthermore,
  $\cR := \ker\varphi$ is not closed and $\cR \oplusdot \cN = \cH_1$,
  which is closed.  Hence $\cH_1$ is pseudo-regular since
  $\cH_1^\bot = \cH_1^0 = \cN$.  While the space $\cR$ is dense in
  $\cH_1$, it is not an operator range, for if it were, as will be
  shown below in
  Theorem~\ref{thm-S_qpr_iff_op_range_and_SplusSperp_closed}, it would
  be regular, and so in particular closed, in which case it would
  equal $\cH_1$.

  By the way, another, more concrete, example of a subspace which is
  not an operator range will be given in
  Example~\ref{exmpl:minimal_or_not}.
\end{example}

\begin{proposition}
  \label{prop:R_oplus_N_qpr_then_R_regular}
  Let $\cS$ be a qpr subspace of a \krein space $\cH$.  If $\cR$ is an
  operator range and $\cS = \cR \oplusdot \cS^0$, then $\cR$ is
  regular.
\end{proposition}

\begin{proof}
  Since $\cS$ is qpr, there is a regular subspace $\tilde{\cR}$ such
  that $\cS = \tilde{\cR} \oplus \cS^0$.  Choose a fundamental
  symmetry $J$ fixing $\tilde{\cR}$ and write $|\cH|$ for the
  associated Hilbert space.

  Let $P$ be the (both Hilbert and \krein space) orthogonal projection
  onto $\tilde{\cR}$.  Then
  $P(\cR) = P(\cR \pdot \cS^0) = P(S) = \tilde{\cR}$.  If $x\in \cR$,
  then $x = \tilde{x} + n$, $\tilde{x} \in \tilde{\cR}$,
  $n \in \cS^0$, so $P(x) = \tilde{x}$.  Hence if $x\in \ker P$, then
  $x\in \cS^0$.  Thus since it is assumed that the sum of $\cR$ and
  $\cS^0$ is direct, $P:\cR \to \tilde{\cR}$ is a bijection.

  Decompose $\tilde{\cR} = \tilde{\cR}^+ \pdot \tilde{\cR}^-$, where
  $\tilde{\cR}^+$ is closed and uniformly positive and $\tilde{\cR}^-$
  is closed and uniformly negative.  Define
  \begin{equation*}
    \cR^\pm = \left\{ x\in \cR : P(x) \in \tilde{\cR}^\pm \right\}.
  \end{equation*}
  If $x\in \cR^\pm$, then $x = \tilde{x} +n$,
  $\tilde{x} \in \tilde{\cR}^\pm$, $n\in \cS^0$, and so
  \begin{equation*}
    \ip{x}{x} = \ip{\tilde{x} +n}{\tilde{x} +n} =
    \ip{\tilde{x}}{\tilde{x}} = \ip{Px}{Px}.
  \end{equation*}
  Hence $P|_{\cR^\pm} : \cR^\pm \to \tilde{\cR}^\pm$ are isometric
  isomorphisms.  Then since $\tilde{\cR}^+$ and $|\tilde{\cR}^-|$ are
  Hilbert spaces, the same is true for $\cR^+$ and $|\cR^-|$.  Similar
  reasoning shows that $\cR^+$ and $\cR^-$ are (\krein space)
  orthogonal, and so $\cR = \cR^+ \pdot \cR^-$ is a \krein space.

  Under the assumption that $\cS$ is qpr,
  Lemma~\ref{lem:S+Sperp_closed} implies that
  $\cR \pdot \cS^\bot = \cS + \cS^\bot$ is closed.  By
  \cite[Theorem~2.3]{FW1971}, $\cR$ an operator range leads to the
  conclusion that $\cR$ is closed in $\cH$, so as desired, $\cR$ is
  regular.
\end{proof}

If a space $\cS$ is pseudo-regular, then $\cS$ is closed by
definition, and by Lemma~\ref{lem:S+Sperp_closed} $\cS + \cS^\bot$ is
closed.  Gheondea proved that these two conditions characterize
pseudo-regularity~\cite{Gheondea1984}.  It is natural to wonder if
there is a similar characterization for qpr spaces.
Lemma~\ref{lem:S+Sperp_closed} says that if $\cS$ is qpr, then
$\cS + \cS^\bot$ is closed.  Also, it is clear from
Lemma~\ref{lem:Lemma0} that the closure of a qpr space is
pseudo-regular.  So this combined with Gheondea's result might lead
one to guess that $\cS + \cS^\bot$ being closed implies that $\cS$ is
qpr. The following example rules this out.

\begin{example}\label{exmpl:S+Sperp_closed_but_not_qpr}
  There is no converse to Lemma~\ref{lem:S+Sperp_closed}; that is,
  simply assuming that $\cS+\cS^\perp$ is closed is not enough to
  guarantee that a space $\cS$ is qpr.  To see this, using the spaces
  constructed in Example~\ref{exmpl:R_not_regular_in_S}, let
  $\cS = \cR$.  As noted there, $\cS^\bot = \cH_1^\bot$, and so
  $\cS + \cS^\bot = \cH_1$, which is closed.  Also, $\cS^0 = \{0\}$
  since $\cS \cap \cN = \{0\}$.  So if $\cS$ were qpr, it would
  necessarily be regular and so closed, giving a contradiction.
\end{example}

While it was seen in the last example that the assumption that
$\cS+\cS^\perp$ closed does not characterize quasi-pseudo-regularity,
the space $\cS+\cS^\perp$ does have some nice properties.

\begin{lemma}
  \label{lem:T_qpr_then_others_pr}
  If $\cS$ is qpr, then $\overline{\cS}$, $\cS^\perp$, and
  $\cS + \cS^\perp$ are pseudo-regular.  Also, if $\cS^\perp$ is
  pseudo-regular and $\cS$ is closed, then $\cS$ is pseudo-regular.
\end{lemma}

\begin{proof}
  If $\cS$ is qpr, then $\cS = \cR \oplus \cS^0$, where $\cR$ is
  regular and $\cS^0$ is neutral.  By regularity of $\cR$, there is a
  fundamental symmetry $J$ such that $J\cR = \cR$, and so $\cR$ and
  $\cS^0$ are also orthogonal in the associated Hilbert space $|\cH|$.
  Notice that $\overline{\cS} = \cR \oplus \cN$,
  $\cN = \overline{\cS^0} = \overline{\cS} \cap \cS^\perp$ by
  Lemma~\ref{lem:Lemma0}.  As $\cN$ is closed and neutral,
  $\overline{\cS}$ is pseudo-regular.  By the choice of $J$,
  $\overline{\cS} = \cR \oplus_{|\cH|} \cN$.  It follows from
  Lemma~\ref{lem:S+Sperp_closed} that $\overline{\cS} + \cS^\perp$ is
  closed.  Since
  $(\overline{\cS} + \cS^\perp)^\perp = \cS^\perp\cap \overline{\cS}=
  \cN$,
  \begin{equation*}
    \cH = (\overline{\cS} + \cS^\perp) \oplus_{|\cH|} J\cN
    =\cR \oplus_{|\cH|} \cN \oplus_{|\cH|} \tilde{\cR}
    \oplus_{|\cH|} J\cN,
  \end{equation*}
  where $\cS^\bot = \cN \oplus_{|\cH|} \tilde\cR$, $\tilde\cR$ is
  closed.  So for $\cL := \cN \oplus_{|\cH|} J\cN$, which is closed,
  $J\cL = \cL$, and so $\cL$ is regular.  This then implies that
  \begin{equation}\label{sep3a2025}
    \cH = (\cR +\tilde\cR)  \oplus_{|\cH|} \cL
    = (\cR +\tilde\cR)\,  \dot{\oplus}_{\cH} \,\cL ,
  \end{equation}
  and regularity of $\cL$ leads to the conclusion that
  $\cR +\tilde\cR$ is also regular.  The sum $\cR +\tilde\cR$ is
  direct, and it is also \krein space orthogonal because
  $\tilde\cR\subseteq\cS^\perp$ and $\cS^\perp\subseteq\cR^\perp$
  (recall $\cR\subseteq\cS$).  Then since $\cR$ is regular, so is
  $\tilde\cR$.

  The last equality in \eqref{sep3a2025} shows that $\cR +\tilde\cR$
  is \krein space orthogonal to $\cL$ and hence to $\cN$.  Thus
  $\cS^\perp = \tilde\cR \oplus \cN$, and since $\cS^\perp$ is closed,
  it is pseudo-regular by Definition \ref{defpseudoreg}.  The equality
  $\cS+\cS^\perp = (\cR +\tilde\cR)\oplus\cN$ similarly shows that
  $\cS+\cS^\perp$ is pseudo-regular since it is closed by Lemma
  \ref{lem:S+Sperp_closed}.  The last statement of the lemma is
  immediate from what has already been shown.
\end{proof}

\begin{corollary}
  \label{cor:SplusSperp_closed_is_pr}
  If $\cS$ is any subspace, then $\overline{\cS+\cS^\perp}$ is
  pseudo-regular.  So if $\cS+\cS^\perp$ is closed, it is
  pseudo-regular, and if the sum is also direct, then it is regular.
\end{corollary}

\begin{proof}
  Suppose that $\cS$ is an arbitrary subspace of a \krein space $\cH$.
  Then since
  \begin{equation*}
    \overline{\cS+\cS^\perp} = \overline{\overline{\cS}+\cS^\perp} =
    (\overline{\cS}\cap\cS^\perp)^\perp,
  \end{equation*}
  and $\cN := \overline{\cS}\cap\cS^\perp$ is a closed neutral
  subspace (so pseudo-regular), it follows that
  $\cN^\perp = \overline{\cS+\cS^\perp}$ is pseudo-regular.  Hence if
  $\cS+\cS^\perp$ is closed, it is pseudo-regular.  If the sum is
  direct, then
  $\overline{\cS^0} = \overline{\cS\cap\cS^\perp} = \{0\}$.  Then by
  Lemma~\ref{lem:Lemma0}, $\cN = \overline{\cS}^0 = \{0\}$.  Hence
  $\cS+\cS^\perp$ is regular.
\end{proof}

The next result is Proposition 3.2.9 in~\cite{Gheondea2022}.  The
proof given here is somewhat different and uses the previous
corollary.

\begin{corollary}
  \label{cor:closed_subsp_in_Pontr_sp_ps-reg}
  Let $\cH$ be a Pontryagin space and $\cS \subseteq \cH$ a closed
  subspace.  Then $\cS$ is pseudo-regular.  Conversely, if $\cH$ is a
  \krein space for which every closed subspace is pseudo-regular, then
  $\cH$ is a Pontryagin space.
\end{corollary}

As is seen in the proof that follows, neutral subspaces of Pontryagin
spaces are finite dimensional, and hence closed.  Consequently,
quasi-pseudo-regular subspaces are pseudo-regular in this context.

\begin{proof}[Proof of Corollary
  $\ref{cor:closed_subsp_in_Pontr_sp_ps-reg}$]
  Assume $\dim \cH^- < \infty$ (the case where $\dim \cH^+ < \infty$
  being handled identically).  Define $\cN = \cS \cap \cS^\bot$, a
  closed neutral subspace.  Observe that
  $(\overline{\cS + \cS^\bot})^0 := (\overline{\cS + \cS^\bot}) \cap
  (\overline{\cS + \cS^\bot})^\bot = \cN$.  By
  Corollary~\ref{cor:SplusSperp_closed_is_pr},
  $\overline{\cS + \cS^\bot} = \cH' \oplusdot \cN$, where $\cH'$ is
  regular.  Note that $\cH'$, as a regular subspace of a Pontryagin
  space, is a Pontryagin space.  Set $\cM = \cS \cap \cH'$ and
  $\cM' = \cS^\bot \cap \cH'$.  Both are closed, they are orthogonal,
  and $\cS + \cS^\bot = (\cM \oplusdot \cM') \oplusdot \cN$.  So
  $\overline{\cM \oplusdot \cM'} = \cH'$, and thus $\cM$ is
  non-degenerate in $\cH'$.  As a closed, non-degenerate subspace of
  the Pontryagin space $\cH'$, $\cM$ is regular (see, for
  example,~\cite[Prop.~3.1.8]{Gheondea2022}).  Thus $\cS$ is
  pseudo-regular.

  For the converse, assume that $\cH$ is a \krein space which is not a
  Pontryagin space, so $\dim \cH^\pm = \infty$.  Define a Hilbert
  space contraction $T \in \cL(\cS^+, \cS^-)$ with the property that
  $\|Tx\| < \|x\|$ for all $x \in \cS^+$ and $\|T\| = 1$.  By the
  Closed Graph Theorem, the graph of $T$ is a closed subspace which is
  strictly positive, but not uniformly positive (so not regular).
\end{proof}

Lemma \ref{lem:T_qpr_then_others_pr} could have been proved using
Gheondea's characterization of pseudo-regular subspaces cited above.
The lemma indicates that pseudo-regular spaces have analogous
properties to regular spaces.  For example, $\cS$ is (pseudo-)regular
if and only if $\cS^\perp$ is (pseudo-)regular, while if $\cS$ is
\mbox{(pseudo-)}regular, then $\cS + \cS^\perp$ is (pseudo-)regular.

By both Lemma~\ref{lem:S+Sperp_closed} and Lemma~\ref{lem:Lemma0}, if
$\cS$ is qpr, then $\cS + \cS^\bot$ is closed and
$\overline{\cS^0} = {\overline{\cS}}^0$.  So one might suspect that
these conditions imply that $\cS$ is qpr.  The next example shows that
this may fail.

\begin{example}
  \label{exmpl:S+Sperp_closed+S_0_dense_but_not_qpr}
  Let $\cH_0$ be an infinite dimensional Hilbert space,
  $\cH = \cH_0 \oplus \cH_0 \oplus \cH_0$.  Make $\cH$ in to a \krein
  space by defining a fundamental symmetry
  \begin{equation*}
    J =
    \begin{pmatrix}
      1_{\cH_0} & 0 & 0 \\
      0 & 0 & 1_{\cH_0} \\
      0 & 1_{\cH_0} & 0
    \end{pmatrix}.
  \end{equation*}
  Let $\cR = \cH_0 \oplus \{0\} \oplus \{0\}$ and
  $\cN = \{0\} \pdot \cH_0 \oplus \{0\}$.  Fix $x \in \cN$,
  $\|x\| = 1$, and set $\cN_0 = \mathrm{span}\,\{x\}$.  The space
  $\cR$ is a Hilbert space (so regular) while $\cN$ is isotropic, so
  $\cN^\bot = \cR \pdot \cN$.  Note too that
  $\cN_0^\bot = \cR \pdot \cN$.  As in the last example, define an
  unbounded linear functional $\varphi$, this time on
  $\cR \pdot \cN_0$, with $\varphi(x) = 1$.  Let $\cS = \ker \varphi$.
  Then $\cH_1 := \cR \pdot \cN_0 = \cS \pdot \cN_0$ is closed and
  $\cS$ is dense in $\cR \pdot \cN_0$.

  Define a bounded operator acting on $\cN$ with dense range $\cW$ not
  containing $x$.  Then ${\cW}^\bot = \cR \pdot \cN$ (as a subspace of
  $\cH$).  Let $\tilde{\cS} = \cS + \cW$.  Then
  \begin{equation*}
    \begin{split}
      {\tilde{\cS}}^\bot
      &= \cS^\bot \cap \cW^\bot\\
      &= (\cR \pdot \cN_0)^\bot \cap (\cR \pdot \cN) \\
      &= \cR^\bot \cap \cN_0^\bot \cap (\cR \pdot \cN) \\
      &= (\cN + J\cN) \cap \cN_0^\bot \cap (\cR \pdot \cN) \\
      &= \cN \cap (\cR \pdot \cN) \\
      & = \cN,
    \end{split}
  \end{equation*}
  which contains $\cN_0$ since $\cW \subseteq \cN$, so
  \begin{equation*}
    \tilde{\cS} + {\tilde{\cS}}^\bot  = (\cS + \cW) + \cN = \cS + \cN
    = \cS \pdot \cN_0 + \cN = \cR + \cN_0 + \cN = \cR \pdot \cN,
  \end{equation*}
  which is closed.  It is easy to see that $\tilde{\cS}^0 = \cW$, so
  ${\overline{\tilde{\cS}^0}} = \cN$.  Also it can be checked that
  ${\overline{\tilde{\cS}}}^0 = \cN$

  Suppose that $\tilde{\cS}$ is qpr.  Then
  $\tilde{\cS} = \tilde{\cR} \pdot \cW$, where $\tilde{\cR}$ is
  regular, and so closed.  By assumption, $\cW$ is an operator range,
  hence in this case $\tilde{\cS}$ is also an operator range, as is
  the closed space $\cH_1$.  By construction,
  $\cS \subseteq \tilde{\cS} \cap \cH_1$.  On the other hand, if
  $x\in \tilde{\cS} \cap \cH_1$, $x = s + n$, $s \in \cS$,
  $n \in \cN_0$.  Thus
  $x - s = n \in \cW \cap \cH_1 = \cW \cap \cN_1 = \{0\}$.  So $n = 0$
  and $x = s \in \cS$.  Hence $\tilde{\cS} \cap \cH_1 = \cS$.
  However, by \cite[Cor.~2,~p.~260]{FW1971}, the intersection of two
  operator ranges is an operator range, and by
  Example~\ref{exmpl:S+Sperp_closed_but_not_qpr} it was seen that
  $\cS$ is not an operator range.  Therefore $\tilde{\cS}$ is not an
  operator range.  Hence $\tilde{\cS}$ is not qpr.
\end{example}

The following is the proper generalization to qpr spaces of Gheondea's
characterization of pseudo-regular spaces, and also contains his
result since closed spaces are operator ranges.

\begin{theorem}
  \label{thm-S_qpr_iff_op_range_and_SplusSperp_closed}
  A subspace $\cS$ of a \krein space $\cH$ is qpr if and only if
  \begin{enumerate}
  \item[(1)] $\cS$ is an operator range, and
    
    \smallskip
  \item[(2)] $\cS+\cS^\perp$ is closed.
  \end{enumerate}
\end{theorem}

\begin{proof}
  If $\cS$ is qpr, the forward implication follows directly by the
  definition of quasi-pseudo-regularity and
  Lemma~\ref{lem:S+Sperp_closed}.

  Now assume $\cS \subseteq \cH$ is an operator range and that
  $\cS+\cS^\perp$ is closed.  By~\cite[Th.~2.4]{FW1971}, there exist
  closed subspaces $\cM_1 \subseteq \cS$, $\cM_2 \subset \cS^\bot$,
  such that
  \begin{equation*}
    \cS+\cS^\perp = \cM_1 \pdot \cM_2.
  \end{equation*}
  Hence $\cS^\perp = \cM_2 \pdot (\cM_1 \cap \cS^\perp)$.

  Fix a fundamental symmetry $J$ and associated Hilbert space $|\cH|$.
  Define on this Hilbert space, the closed space
  \begin{equation*}
    \cR := \cM_1 \ominus_{|\cH|} (\cM_1 \cap \cS^\perp).
  \end{equation*}
  Then $\cR \subseteq \cM_1 \subseteq \cS$,
  $\cM_1 = \cR \pdot (\cM_1 \cap \cS^\perp)$, and
  \begin{equation*}
    \cS+\cS^\perp = \cM_1 \pdot \cM_2 = \cR \pdot (\cM_1 \cap \cS^\perp)
    \pdot \cM_2 = \cR \pdot \cS^\perp,
  \end{equation*}
  where the latter sum is direct since
  $\cR \cap \cS^\perp = \cR \cap (\cM_1 \cap \cS^\perp) = \{0\}$.
  Hence
  \begin{equation*}
    \cS = \cS \cap (\cR \pdot \cS^\perp) = \cR \pdot \cS^0.
  \end{equation*}

  Let $Q$ be the (not necessarily orthogonal) projection on
  $\cS+\cS^\perp$ with range $\cS^\perp$ and kernel $\cR$.  This is
  bounded since $\cS+\cS^\perp = \cR \pdot \cS^\perp$.  As
  $\overline{\cS^0} \subseteq \cS^\perp$ is closed, by continuity, the
  pre-image $Q^{-1}(\overline{\cS^0}) = \overline{\cS^0} + \cR$ is
  closed and contains $\cS^0 \pdot \cR$.  Hence
  $Q^{-1}(\overline{\cS^0}) \supseteq \overline{\cS^0 \pdot \cR}$,
  while the reverse inclusion always holds.  Therefore,
  \begin{equation*}
    \overline{S} = \overline{\cR \pdot \cS^0} = \overline{\cS^0} \pdot
    \cR.
  \end{equation*}

  Arguing as above with $\overline{\cS}$ in place of $\cS$, one has
  \begin{equation*}
    \overline{\cS} = (\cR \pdot \cS^\perp) \cap \overline{\cS} = \cR +
    (\overline{\cS}\cap\cS^\perp) = \cR \pdot (\overline{S})^0.
  \end{equation*}
  Consequently, since $\overline{\cS^0} \subseteq \overline{S}^0$,
  \begin{equation*}
    \cN := \overline{\cS^0} = \overline{S}^0.
  \end{equation*}
  The space $\cN \subseteq \cS^\bot$, so the sum is both orthogonal
  and direct; that is,
  \begin{equation*}
    \overline{\cS} = \cR \pdot \cN.
  \end{equation*}
  
  By parallel arguments with
  $\tilde R = \cM_2 \ominus_{|\cH|} (\cM_2 \cap \overline{\cS})$, one
  has $\tilde R$ closed and
  \begin{equation*}
    \cS^\perp = \tilde{\cR} \pdot \cN.
  \end{equation*}
  Then
  \begin{equation*}
    \cS + \cS^\perp =  \overline{S} +\cS^\perp = (\cR+\tilde\cR) + \cN.
  \end{equation*}
  By Corollary \ref{cor:SplusSperp_closed_is_pr}, $\cS + \cS^\perp$ is
  qpr, and $\cR+\tilde\cR$ is an operator range because it is the sum
  of operator ranges.  Since $\cR$ and $\tilde\cR$ are orthogonal
  to~$\cN$, $\cR+\tilde\cR$ is orthogonal to~$\cN$.  If
  $\cN \ni n = r + r'$, with $r\in \cR$ and $r' \in \tilde{\cR}$, then
  $r = -r' + n \in \cS^\bot$.  As $\cR \cap \cS^\bot = \{0\}$,
  conclude that $r = 0$ and $r' = n$.  Also,
  $\tilde{\cR} \cap \cN = \{0\}$, so $r' = n = 0$, meaning that
  $(\cR +\tilde{\cR}) \cap \cN = \{0\}$.  Therefore by Proposition
  \ref{prop:R_oplus_N_qpr_then_R_regular}, $\cR+\tilde\cR$ is regular.
  
  The sum $\cR+\tilde\cR$ is also orthogonal and direct.  Since the
  sum is regular, it follows that $\cR$ and $\tilde\cR$ are regular.
  The decomposition $\cS = \cR \pdot \cS^0$ obtained above is also
  orthogonal because $\cR\subseteq\cS$ and $\cS^0\subseteq\cS^\perp$.
  Therefore $\cS$ is qpr.
\end{proof}

\begin{proposition}\label{oct28b2024}
  Let $\cS_1$ and $\cS_2$ be qpr subspaces of a \krein space such that
  $\cS_1\perp\cS_2$.  Then $\cS_1+\cS_2$ is qpr.
\end{proposition}

\goodbreak

\begin{proof}
  Let $\cS_i=\cR_i\pdot\cN_i$, where $\cR_i$ is regular and $\cN_i$ is
  neutral, $i=1,2$.  Then $\cR_1,\cR_2,\cN_1,\cN_2$ are pairwise
  orthogonal.  Thus $\cS_1+\cS_2=\cR + \cN$, where $\cR=\cR_1+\cR_2$
  is regular, $\cN=\cN_1+\cN_2$ is neutral, and $\cR\perp\cN$.  If
  $x\in\cR\cap\cN$, then $x\in\cR$ and $x\perp\cR$; hence $x=0$ by the
  nondegeneracy of a regular subspace.  So $\cR\cap\cN = \{0\}$, and
  therefore $\cS_1+\cS_2=\cR\pdot\cN$ is qpr.
\end{proof}

Of course this last proposition can be extended to finite families of
qpr spaces.  The extension to infinite families is explored in
Theorem~\ref{qpr_sums}.

\section{Sums of operator ranges}
\label{sec:sums-operator-ranges}

Quasi-pseudo-regular spaces are characterized in terms of operator
ranges, and closed subspaces are obviously also operator ranges.  For
this reason, the study of orthogonal sums of regular, pseudo-regular,
and qpr spaces, leads naturally to the study of sums of operator
ranges.  Finite sums of operator ranges are operator
ranges~~\cite[Th.~2.2]{FW1971}.  The consideration of infinite sums of
such spaces is more subtle.  The main issues here are already present
for Hilbert spaces, but our applications involve \krein spaces.

Attention therefore turns to sums of operator ranges
$\cM_\lambda = \ran T_\lambda$ acting on copies of a fixed \krein
space $\cH$ indexed by $\Lambda$.  By Definition \ref{defbigoplus},
such a sum $\cM$ is an operator range such that
\begin{equation*}
  \mathrm{span}_\lambda\, \cM_\lambda \subseteq \cM \subseteq
  \sum^{\scalebox{0.5}{\text{MS}}}_\lambda \cM_\lambda.
\end{equation*}
In general this will not be unique (see Example
\ref{exmpl:minimal_or_not}).  One way of constructing such a $T$ is as
an operator from a Cartesian product space into $\cH$.  If $\cH$ is a
\krein space and $\Lambda$ is an index set, define the \krein space
$\cH^\Lambda$ as in \cite[Section 4]{Rovnyak2025}.  For any associated
Hilbert space $|\cH|$ for $\cH$, $|\cH|^\Lambda$ is an associated
Hilbert space for~$\cH^\Lambda$.

Our first condition for the existence of a sum is based on the Cauchy
criterion.  As in \cite{Rovnyak2021}, given a family
$\{a_\lambda\}_\linL$ of non-negative scalars, the notation
$\sum_\linL a_\lambda <\infty$ means that all finite sums from this
set have a uniform upper bound.  In this case, $\sum_\linL a_\lambda$
is Moore-Smith convergent, and its Moore-Smith sum (also denoted
$\sum_\linL a_\lambda$) is the supremum of all finite sums.  Moreover
the Cauchy criterion is satisfied, and only countably many summands
are non-zero (see Lemma~2.2 and Theorem~2.3 of \cite{Rovnyak2021}).

\begin{theorem}
  \label{thm:ctble_sum_op_ranges}
  Let $\{\cM_\lambda\}_{\linL}$ be a collection of operator ranges in
  a \krein space $\cH$, and let $\{T_\lambda\}_{\linL}$ be operators
  in $\lh$ with $\ran T_\lambda = \cM_\lambda$ for each $\linL$.  If
  $\sum_\linL \|T_\lambda\| < \infty$ relative to some and hence any
  norm for~$\cH$, then the formula
  \begin{equation}\label{aug20b}
    Tx = \sum^{\scalebox{0.5}{\text{MS}}}_{\linL} x_\lambda,
    \quad x=\{x_\lambda\}_\linL \in \cH^\Lambda,
  \end{equation}
  defines an operator in $\cL(\cH^\Lambda,\cH)$ whose range is a sum
  of $\{\cM_\lambda\}_{\linL}$.
\end{theorem}

When $\Lambda$ is countable, the condition
$\sum_\linL \|T_\lambda\| < \infty$ in Theorem
\ref{thm:ctble_sum_op_ranges} can always be arranged by scaling, that
is, by multiplying the operators $T_\lambda$ by positive constants.
Scaling does not change the relations $\ran T_\lambda = \cM_\lambda$,
$\linL$.

\begin{corollary} \label{thm:ctble_sum_op_ranges_cor} For every
  countable collection $\{\cM_\lambda\}_{\linL}$ of operator ranges in
  a \krein space, there is a bounded operator $T$ with range a sum of
  $\{\cM_\lambda\}_{\linL}$.
\end{corollary}

\begin{proof}[Proof of Theorem $\ref{thm:ctble_sum_op_ranges}$]
  Choose and fix an associated Hilbert space $|\cH|$ and norm
  $\|\cdot\|$ for $\cH$.  By assumption the scalars
  $\{\|T_\lambda\|\}_\linL$ are Moore-Smith summable.  By the Cauchy
  criterion, for every $\varepsilon>0$ there is a finite set
  $\cF_0\subseteq\Lambda$ such that
  $\sum_{\lambda\in\cF} \|T_\lambda\|<\varepsilon$ for every finite
  set $\cF\subseteq\Lambda$ disjoint from $\cF_0$.  Consider some
  fixed $x=\{x_\lambda\}_\linL$ in $\cH^\Lambda$.  Given
  $\varepsilon>0$ there is a finite set $\cF_{0,x}\subseteq\Lambda$
  such that
  \begin{equation*}
    \sum_{\lambda\in\cF} \|T_\lambda x\|<\varepsilon/\|x\|
  \end{equation*}
  for every finite set $\cF\subseteq\Lambda$ disjoint from
  $\cF_{0,x}$.  Then for all such $\cF$,
  \begin{equation}\label{aug22a}
    \|\sum_{\lambda \in\cF} T_\lambda x_\lambda\| \leq \sum_{\lambda
      \in\cF} \|T_\lambda x_\lambda\| \leq \sum_{\lambda \in\cF}
    \|T_\lambda\| \|x_\lambda\| \leq \|x\| \sum_{\lambda \in\cF}
    \|T_\lambda\| < \varepsilon.
  \end{equation}
  It follows that the family $\{T_\lambda x_\lambda\}_\linL$
  Moore-Smith summable \cite[Theorem A.2(3)]{Rovnyak2025}.  Since $x$
  is arbitrary, the correspondence
  $T\colon x \to \sum^{\scalebox{0.4}{\text{MS}}}_{\linL} T_\lambda
  x_\lambda$ is an everywhere defined mapping from $\cH^\Lambda$ into
  $\cH$.  The mapping is clearly linear.  The first three inequalities
  in \eqref{aug22a} hold for every finite subset $\cF$ of~$\Lambda$,
  and the sums $\sum_{\lambda\in\cF} \|T_\lambda \|$ in the fourth
  term of \eqref{aug22a} have a uniform upper bound.  Thus $T$ is
  bounded and hence continuous.  It is immediate from the formula
  defining $T$ that $\ran T$ is a sum of $\{\cM_\lambda\}_{\linL}$.
\end{proof}

The same conclusion holds with a weaker condition on operator norms.

\begin{theorem} \label{thm:unctble_sum_op_ranges} Let
  $\{\cM_\lambda\}_{\linL}$ be a collection of operator ranges in a
  \krein space $\cH$, $\{T_\lambda\}_{\linL}$ operators in $\lh$ with
  $\ran T_\lambda = \cM_\lambda$ for each $\linL$.  The following are
  equivalent:
  \begin{enumerate}
  \item[(1)] For some and hence any norm for $\cH$ there is a constant
    $C > 0$ such that $ \| \sum_{\lambda\in\cF} T_\lambda\| \leq C$
    for every finite $\cF \subseteq \Lambda$.
  \item[(2)] The formula
    \begin{equation}\label{aug20bb}
      Tx = \sum^{\scalebox{0.5}{\text{MS}}}_{\linL} T_\lambda x_\lambda, 
      \quad x=\{x_\lambda\}_\linL \in \cH^\Lambda,
    \end{equation}
    defines an operator in $\cL(\cH^\Lambda,\cH)$.
  \end{enumerate}
  In this case $\ran T$ is a sum of $\{\cM_\lambda\}_{\linL}$.
\end{theorem}

\begin{proof}
  Assume (1).  Choose and fix an associated Hilbert space $|\cH|$ and
  norm $\|\cdot\|$ for $\cH$.  Then $|\cH|^\Lambda$ is an associated
  Hilbert space for $\cH^\Lambda$.  An equivalent form of our
  assumption is that $\| \sum_{\lambda\in\cF} T_\lambda^*\| \leq C$
  for every finite $\cF \subseteq \Lambda$.  As in the proof of
  \cite[Theorem 3.1]{Rovnyak2021}, this implies that
  $ \sum_{\lambda\in\cF} \|T_\lambda^* u\|^2 \leq 4C^2 \|u\|^2 $,
  $u\in\cH$, for all finite $\cF\subseteq\Lambda$, and thus
  \begin{equation*}
    \sum^{\scalebox{0.5}{\text{MS}}}_{\linL} \|T_\lambda^* u\|^2 \leq
    4C^2 \|u\|^2, \qquad u\in\cH.
  \end{equation*}
  Hence there is an operator $W\in\cL(|\cH|,|\cH|^\Lambda)$ such that
  \begin{equation*}
    Wu = \{T_\lambda^*u\}_\linL , \qquad u\in\cH.
  \end{equation*}
  Then also $W\in\cL(\cH,\cH^\Lambda)$, and
  $W^*\in\cL(\cH^\Lambda,\cH)$.
 
  Suppose $x=\{x_\lambda\}_\linL \in \cH^\Lambda$.  Fix
  $\mu\in\Lambda$, and let
  $x^{(\mu)} = \{x^{(\mu)}_\lambda\}_\linL \in\cH^\Lambda$, where
  $x^{(\mu)}_\mu = x_\mu$ and all other terms are zero.  Then
  \begin{equation*}
    W^* x^{(\mu)} = T_\mu x_\mu.
  \end{equation*}
  For if $u$ is any element of~$\cH$, then
  \begin{align*} {\ip{W^*x^{(\mu)}}{u}}_\cH &=
    {\ip{x^{(\mu)}}{Wu}}_{\cH^\Lambda}
    = {\ip{ \{x^{(\mu)}_\lambda\}_\linL}{\{T_\lambda^* u\}_\linL}}_{\cH^\Lambda}\\
    &\hskip1cm =
    {\ip{x^{(\mu)}_\mu}{T_\mu^* u}}_{\cH} = {\ip{T_\mu x_\mu}{u}}_{\cH}.
  \end{align*}
  Since $x$ is the Moore-Smith sum of $\{x^{(\mu)}\}_{\mu\in\Lambda}$,
  and $W^*$ is continuous, $W^*x$ is the Moore-Smith sum of
  $W^* x^{(\mu)} =T_\mu x_\mu$, $\mu\in\Lambda$, in~$\cH$.  Renaming
  $W^*$ as $T$, \eqref{aug20bb} is obtained.  It is immediate from
  \eqref{aug20bb} that $\ran T$ is a sum of the family
  $\{\cM_\lambda\}_{\linL}$.

  Conversely, if (2) holds, then for each $f$ in $\cH$ any finite sum
  of vectors $\{T_\lambda f\}_\linL$ is uniformly bounded (see
  Theorem~A.2 of~\cite{Rovnyak2025}).  By the Uniform Boundedness
  Principle, the finite sums of operators $\{T_\lambda \}_\linL$ have
  a uniform bound; that is, (1) holds.
\end{proof}

Here are a few examples illustrating aspects of the last theorems.

\goodbreak

\begin{example}
  \label{exmpl:span_unctbly_many_op_ranges_may_not_be_an_op_range}
  As observed in Example~\ref{exmpl:R_not_regular_in_S} and subsequent
  examples, in a separable Hilbert space $\cH$ there is a space $\cR$
  which is \emph{not} an operator range.  For each $x \in \cR$ with
  $\|x\| = 1$, set $\cM_x = \mathrm{span}\, x$.  Being one
  dimensional, these spaces are obviously the ranges of projection
  operators, and $\cR = \mathrm{span}_x\, M_x$.  Of course,
  $\bigvee_x \cM_x = \overline{\cR}$ is the range of the projection
  onto $\overline{\cR}$.  It is an open question whether there is an
  operator range $\cM$ containing $\cR$ and properly contained in
  $\overline{\cR}$.
\end{example}

\begin{example}
  \label{exmpl:minimal_or_not}
  On the Hilbert space $\cH = \mathbb C^\bbN (= \ell^2)$, consider the
  Toeplitz operator
  \begin{equation*}
    T :=
    \begin{pmatrix}
      2 & 1 & 0 & \cdots \\
      1 & \ddots & \ddots & \ddots \\
      0 & \ddots & \ddots & \ddots \\
      \vdots & \ddots & \ddots & \ddots
    \end{pmatrix}.
  \end{equation*}
  This is easily seen to be positive with dense, non-closed range.
  Write $c_k$ for the $k$th column of $T$, and $\cM_k$ for the span of
  $c_k$.  Then the elements of the family $\{\cM_k\}_k$ are pairwise
  disjoint and $\mathrm{span}_k\, \cM_k \subseteq \ran T$.  Let
  $\{e_k\} \subset \cH$ to be the standard basis for $\cH$, so that
  $c_k = T e_k$.  Evidently, any $x\in \cH$ is a Moore-Smith sum,
  $x = \sum^{\scalebox{0.4}{\text{MS}}}_k x_k e_k$, and so
  $Tx = \sum^{\scalebox{0.4}{\text{MS}}}_k x_k Te_k =
  \sum^{\scalebox{0.4}{\text{MS}}}_k x_k c_k$.  Thus
  $\ran T \subseteq \sum^{\scalebox{0.4}{\text{MS}}}_k \cM_k$, and so
  $\ran T$ is a sum of $\{\cM_k\}$.  (With $T_j = TP_j$, $P_j$ the
  selfadjoint projection onto $\mathrm{span}\,e_j$, $j \geq 1$, this
  illustrates Theorem~\ref{thm:unctble_sum_op_ranges}).

  Let $h \in
  \cH$ be any vector with infinite support; that is, with infinitely
  many non-zero entries.  Define the diagonal operator
  $D_h$ with $(n,n)$-entry equal to $h_n$ if $h_n \neq 0$ and
  $1$ otherwise.  Then $\ran D_h$ is dense in $\cH$ and $h\notin \ran
  D_h$.  Define $T_h = TD_h$.  Thus $Th \notin \ran
  T_h$.  Note that $T_h e_k$ spans
  $\cM_k$.  Consequently, $\mathrm{span}_k\, \cM_k \subseteq \ran T_h
  \subset \ran T$, where the second containment is proper.  So $\ran
  T_h$ is a sum of $\{\cM_k\}$.

  By \cite[Cor.~2,~p.~260]{FW1971}, the intersection of any two
  operator ranges is an operator range.  So it is natural to wonder if
  there is a bounded operator $\tilde{T}$ with minimal range, in the
  sense that its range is a sum of $\{\cM_k\}$ contained in any other
  sum of these spaces.  There is no loss in generality in assuming
  that, if it exists, $\tilde{T} \geq 0$.  By construction,
  $\ran \tilde{T} \subseteq \ran T_h$ for all $h$ with infinite
  support, and so $Th \notin \ran\tilde{T}$ for all such $h$.

  For each $k$, choose $\tilde{e}_k \in \cH$ such that
  $\tilde{T}\tilde{e}_k = c_k$, and set
  $\alpha_k = 1/(2^k\|\tilde{e}_k\|$).  Then
  $f = \sum_k \alpha_k \tilde{e}_k \in \cH$.  Suppose that
  $\tilde{T} f$ has finite support.  Then there is $n$ such that for
  all $k > n$, the $k$th entry of $\tilde{T} f$ is $0$.  This means
  that $\alpha_{k+1} = -(\alpha_{k-1} + 2\alpha_{k}) < 0$,
  contradicting the fact that $\alpha_k > 0$ for all $k$.  Hence
  $\tilde{T} f$ has infinite support.  Since
  $\ran \tilde{T} \subseteq \ran T$, $\tilde{T} f = Th$ for some
  $h \in \cH$.  This $h$ has infinite support since otherwise $Th$
  would have finite support.  Recall that for $T_h$, $\ran T_h$ is a
  sum of $\{\cM_k\}_k$ with $h \notin \ran T_h$.  By assumption,
  $\ran \tilde{T} \subseteq \ran T_h$, implying that
  $Th \in \ran T_h$, giving a contradiction.

  Hence there is no minimal operator range $\cM$ which is a sum of
  $\{\cM_k\}_k$.  Moreover, this implies that
  $\mathrm{span}_k\, \cM_k$ is a subspace of $\cH$ which is \emph{not}
  an operator range.
\end{example}

\begin{example}
  \label{exmpl:unctble_sum_of_same_space}
  In this example, the set $\{\cM_\lambda\}_{\linL}$ is taken to be be
  an infinite collection of operator ranges in a Hilbert space $\cH$
  where $\cM_\lambda = \cM \neq \{0\}$ for all $\lambda$.  In this
  case, $\mathrm{span}_\lambda \cM_\lambda = \cM$ so $\cM$ is a sum of
  these operator ranges.

  Alternately, $\overline{\cM}$ is also a sum of these spaces.
  Obviously, $\overline{\cM}$ is an operator range since it is the
  range of a projection.  Now suppose that
  $x\in \overline{\cM}\backslash \cM$.  For $n \in \bbN$, choose
  $x_n \in \cM$ such that $x_0 = 0$ and $z_n = x- x_n $ satisfies
  $\|z_n\| = \|x - x_n\| \le \|x\|/2^n$.  Set
  $y_n = x_{n} - x_{n-1} = z_{n-1} - z_n \in \cM$, $n\ge 1$.  Then
  $\|y_n\| \le \left(\tfrac{1}{2^n} + \tfrac{2}{2^n}\right) \|x\| =
  \tfrac{3}{2^n} \|x\|$ and $y_1 +y_2 +\cdots+ y_n = x_n \to x$ as
  $n\to\infty$.  Hence $\{y_n\}$ is absolutely summable with
  $\sum \|y_n\| \leq 3\|x\|$.  By the Cauchy criterion, $\{y_n\}$ is
  Moore-Smith summable, and $\sum^{\scalebox{0.4}{\text{MS}}} y_n = x$.
  In other words,
  $\sum^{\scalebox{0.4}{\text{MS}}} \cM = \overline{\cM}$, and so
  $\overline{\cM}$ is a sum of $\{\cM_\lambda\}_{\linL}$.

  Note that Example~6.3 of~\cite{Rovnyak2025} shows more generally
  that in an orthogonal direct sum of \krein spaces, the closed span
  and Moore-Smith may differ.
\end{example}

\begin{example}
  \label{exmpl:T_may_not_be_closable}
  Ranges of closed operators are operator ranges in the usual sense.
  Unfortunately, if a condition like that required in
  Theorem~\ref{thm:unctble_sum_op_ranges} does not hold, then the
  operator $T$ constructed there will generally not be closable, and
  so the approach taken in the proof of
  Theorem~\ref{thm:unctble_sum_op_ranges} does not appear to
  immediately generalize.  Here is a simple example illustrating the
  problem.

  Let $\cH = \mathbb C$, and suppose $T$ on $\cH^\bbN$ to $\cH$ is the
  row operator with all entries~$1$.  Vectors in $\cH^\bbN$ which have
  only finitely many non-zero entries (the set of which is dense) are
  in the domain of $T$.  Thus the vector $x_n$ with the first $n$
  entries equal to $1/n$ is in the domain of $T$ and
  $\|x_n\|^2 = 1/n$.  So $x_n \to 0$ in norm with $n$.  However,
  $Tx_n = 1$ for all $n$, meaning that $T$ is not closable.

  Admittedly, this example could have been dealt with since the number
  of operators is countable, but uncountable variations are easily
  constructed.
\end{example}

\goodbreak

\section{Convergence of nets of projections}
\label{sec:conv-nets-proj}

A more relaxed notion of projection is used when considering sums of
pseudo-regular spaces, as the selfadjoint projections which suffice
for regular spaces are no longer enough here.  On a \krein space
$\cH$, call $P\in\lh$ a \textbf{projection} if $P^2=P$.  If $P$ is a
projection, so is $1-P$; moreover, $\ran P=\ker(1-P)$, and
$\ker P = \ran(1-P)$.  The range of a projection is closed, and every
closed subspace of $\cH$ is the range of a projection.  If
$\cM=\ran P$ and $\cN=\ker P$ for some projection $P\in\lh$, then
$\cH = \cM\pdot\cN$.  Conversely, any pair of closed subspaces $\cM$
and $\cN$ such that $\cH = \cM\pdot\cN$ has this form for a unique
projection $P\in\lh$.

There is an even more general notion of projection on \krein spaces
allowing for unbounded operators or even linear relations; see
\cite{BHdS} and~\cite{Arias-etal2023}.  While interesting in the
context of quasi-pseudo-regular spaces, it is not addressed here.

\begin{lemma}[\!\!{\cite[p.~481]{DS-I}}]
  \label{projorder}
  Let $P_1,P_2\in\lh$ be projections, and set $\cM_j=\ran P_j$ and
  $\cN_j=\ker P_j$, $j=1,2$.  Then
  \begin{itemize}
  \item[(1)] $\cM_1\subseteq\cM_2$ if and only if $P_2P_1=P_1$;
  \item[(2)] $\cM_1\subseteq\cM_2$ and $\cN_1\supseteq\cN_2$ if and
    only if $P_2P_1=P_1P_2=P_1$.
  \end{itemize}
\end{lemma}

It is also known that a closed subspace $\cM$ of a \krein space $\cH$
is pseudo-regular if and only if $\cM = \ran Q$, where $Q\in\lh$ is a
\textbf{normal projection}, that is, $Q^*Q=QQ^*$.  This is proved by
Maestripieri and Per{\'{\i}}a in Theorem~4.3 of
\cite{MaestripieriEtAl2013} using the characterization of
pseudo-regular subspaces in
Theorem~\ref{thm-S_qpr_iff_op_range_and_SplusSperp_closed}.  An
important difference with the regular case is that the normal
projection $Q$ is not determined by~$\cM$.  Observe that for any
normal projection~$Q$, $Q^*$ is a normal projection.

\begin{theorem}\label{ThA}
  Let $\cH$ be a \krein space, $\{P_d\}_\dinD$ a net of projections in
  $\lh$ such that
  \begin{equation}\label{ThAa}
    P_dP_{d'} = P_{d'}P_d = P_d \quad\text{whenever}\quad d'\ge d.
  \end{equation}
  Assume that relative to some and hence any associated norm
  $\|\cdot\|$ for~$\cH$, there is a $C>0$ such that
  \begin{equation}\label{ThAb}
    \|P_d\|\le C,\qquad \dinD.
  \end{equation}
  Set $\cM_d=\ran P_d$, $\cN_d=\ker P_d$, $\dinD$, and
  \begin{equation}\label{ThAc}
    \cM=\bigvee_{\dinD} \cM_d, \qquad \cN = \bigcap_\dinD\cN_d.
  \end{equation}
  Then there is a projection $P\in\lh$ such that
  \begin{equation}\label{ThAd}
    \lim_\dinD P_d f = Pf, \qquad f\in\cH,
  \end{equation}
  in the strong topology of~$\cH$.  Moreover, $\ran P=\cM$ and
  $\ker P = \cN$.
\end{theorem}

By Lemma \ref{projorder}, the condition \eqref{ThAa} implies that
\begin{equation}\label{ThAaa}
  \ran P_{d'}\supseteq\ran P_{d}\quad\text{and}\quad
  \ker P_{d'}\subseteq\ker P_{d},
\end{equation}
whenever $d'\ge d$.  If the projections $P_d$, $d\in D$, commute, then
conversely \eqref{ThAaa} implies \eqref{ThAa}.

\begin{proof}[Proof of Theorem $\ref{ThA}$]
  Without loss of generality, assume that $\cH$ is a Hilbert space.
  This can be done since the projections $\{P_d\}_\dinD$ satisfy the
  same conditions relative to any associated Hilbert space $|\cH|$,
  and strong convergence in $\cH$ is the same as strong convergence in
  $|\cH|$.

  First show that $\cH=\cM+\cN$.  Fix $f\in\cH$.  For each $\dinD$,
  $P_d$ is a projection, and hence $\cH=\cM_d\pdot\cN_d$.  Write
  \begin{equation*}
    f = u_d + v_d,\qquad u_d\in\cM_d,\;v_d\in\cN_d.
  \end{equation*}
  By \eqref{ThAb}, each $u_d=P_d f, \dinD,$ is in the closed ball
  centered at $0$ of radius $C\|f\|$.  Since the closed ball is
  compact in the weak topology of~$\cH$, there is a subnet
  $\{u_{d_i}\}_{i\in E}$ of $\{u_d\}_{\dinD}$ that converges weakly to a
  vector $u\in\cH$ such that $\|u\|\le C\|f\|$.  Thus
  \begin{equation*}
    \lim_{ i\in E} \ip{u_{d_i}}{x} = \ip{u}{x},\quad x\in\cH.
  \end{equation*}
  Set $v=f-u$.  Then $\{v_{d_i}\}_{i\in E}$ is a subnet of
  $\{v_d\}_{\dinD}$ such that
  \begin{equation*}
    \lim_{ i\in E} \ip{v_{d_i}}{x} = \ip{v}{x},\quad x\in\cH,
  \end{equation*}
  that is, $v=\lim_{i\in E} v_{d_i}$ weakly.

  \smallskip
  \noindent \textbf{Claim:} $v\in\cN_d$ for each $\dinD$.  \smallskip

  Fix $d\in D$.  By the definition of a subnet, there is an $i_0$ in
  $E$ such that $i\ge i_0 \Rightarrow d_i\ge d$.  By replacing the net
  $\{v_{d_i}\}_{i\in E}$ by its truncation at~$i_0$ (that is,
  $\{v_{d_i}\}_{i\in E \text{,\,} i\ge i_0}$), it can be assumed
  that $i\ge i_0$ for all $i\in E$.  Then for all $i\in E$,
  $d_i\ge d$, hence
  \begin{equation*}
    \cN_{d_i} = \ker P_{d_i}\subseteq \ker P_{d} = \cN_d.
  \end{equation*}
  Therefore $v_{d_i} \in\cN_{d}$ for all $i\in E$, and hence
  $v \in \cN_d$ since a norm closed convex set is weakly closed.  The
  claim follows.

  \smallskip By the Claim, $v\in\cN$.  Thus each $f$ in $\cH$ has the
  form $f=u+v$ with $u\in\cM$ and $v\in\cN$.  Thus $\cH=\cM+\cN$.

  \smallskip Next that $\cM\cap\cN=\{0\}$ and hence $\cH=\cM\pdot\cN$.
  The key step is to show that $\lim_\dinD P_df=f$ in norm for every
  $f\in\cM$.  By \eqref{ThAc}, $\cM$ is the closure of the set of all
  finite sums of vectors in the subspaces $\cM_d$, $\dinD$.  Suppose
  $h=h_{d_1}+ \cdots h_{d_r}$, where $h_i\in\cM_{d_i}$, $i=1,\dots,r$.
  By the definition of a net, there is a $d$ in $D$ such that
  $d\ge d_i$, $i=1,\dots,r$, and then $h\in\cM_d$ by \eqref{ThAaa}.
  It follows that
  \begin{equation*}
    \cM = \overline{\bigcup_\dinD \cM_d}.
  \end{equation*}
  Hence for any $f\in\cM$ and $\varepsilon>0$, there is a $d_0\in D$
  and a vector $g\in\cM_{d_0}$ such that
  \begin{equation*}
    \|f-g\| < \frac{\varepsilon}{C+1},
  \end{equation*}
  where $C$ is the constant in \eqref{ThAb}.  For any $d\ge d_0$,
  $P_d\, g=g$ because $\cM_d\supseteq\cM_{d_0}$ by \eqref{ThAaa}.
  Hence for all $d\ge d_0$,
  \begin{align*}
    \|f-P_d f\| &= \|f-g + P_d(g-f)\| \\
                &\le \|f-g\| + \|P_d\| \|g-f\| \\
                &< \frac{\varepsilon}{C+1} + \frac{C \varepsilon}{C+1}\\
                &= \varepsilon,
  \end{align*}
  which means that $\lim_\dinD P_df=f$.  Now if both $f\in\cM$ and
  $f\in\cN$, then $\lim_\dinD P_df=f$ by what was just showed, and
  also $\lim_\dinD P_df=0$ because $f\in\cN$ and
  $\cN\subseteq\cN_d = \ker P_d$ for each $\dinD$; hence $f=0$.  It
  follows that $\cM\cap\cN=\{0\}$ and hence $\cH=\cM\pdot\cN$.

  Since $\cH=\cM\pdot\cN$, there is a unique projection $P\in\lh$ with
  range $\cM$ and kernel $\cN$.  The preceding argument incidentally
  shows that \eqref{ThAd} holds separately for $f\in\cM$ and
  $f\in\cN$, and hence it holds for all $f\in\cH$.  All of the
  statements in the theorem follow.
\end{proof}

\begin{theorem}\label{ThB}
  In Theorem $\ref{ThA}$ assume that the projections $P_d$, $\dinD$,
  are all normal.  Then $P$ is normal.  Moreover,
  \begin{equation}\label{ThBa}
    \lim_\dinD P_d^* g = P^*g,\qquad g\in\cH,
  \end{equation}
  in the strong topology of~$\cH$, and
  \begin{equation}\label{ThBb}
    \ran P^* = \bigvee_\dinD \ran P_d^*,
    \quad
    \ker P^* = \bigcap_\dinD \ker P_d^*.      
  \end{equation}   
\end{theorem}

In the special case that each $P_d$, $\dinD$, is selfadjoint, $P$ is
selfadjoint and \eqref{ThBa} and \eqref{ThBb} reduce to \eqref{ThAc}
and \eqref{ThAd}.  This special case also follows from Theorem 5.2.6
in \cite{Gheondea2022} and provides another way to prove the key
implication (2) $\Rightarrow$ (3) in Theorem \ref{maintheorem1}.  The
authors thank the anonymous referee who called attention to the paper
\cite{Gheondea1988} and \cite[Theorem 5.2.6]{Gheondea2022}, both of
which treat this special case, the former for sequences and the latter
for arbitrary nets.  Limits of nondegenerate subspaces are also
discussed in \cite{Gheondea1988, Gheondea2022}, but this is outside of
the scope of this work.

\begin{proof}
  The projections $P_d^*$, $\dinD$, satisfy \eqref{ThAa} and
  \eqref{ThAb} with each $P_d$ replaced by $P_d^*$.  Set
  $\widetilde \cM_d=\ran P_d^*$, $\widetilde\cN_d=\ker P_d^*$,
  $\dinD$, and
  \begin{equation}
    \label{ThAdd}
    \widetilde\cM=\bigvee_{\dinD} \widetilde\cM_d,
    \qquad \widetilde\cN = \bigcap_\dinD \widetilde\cN_d.
  \end{equation}
  By Theorem \ref{ThA} there is a projection $\widetilde P\in\lh$ such
  that
  \begin{equation}\label{ThAcc}
    \lim_\dinD P_d^* f = \widetilde P f, \qquad f\in\cH,
  \end{equation}
  strongly, and $\ran \widetilde P= \widetilde \cM$ and
  $\ker \widetilde P = \widetilde\cN$.  Since
  \begin{equation*}
    \ip{P_df}{g} = \ip{f}{P_d^* g}, \qquad \dinD,
  \end{equation*}
  in the limit one obtains $\ip{Pf}{g} = \ip{f}{\widetilde P g}$, and
  therefore $\widetilde P=P^*$.  Thus \eqref{ThBa} and \eqref{ThBb}
  follow from \eqref{ThAcc} and \eqref{ThAdd}.
  
  It remains to show that $P$ is normal.  By assumption, each $P_d$ is
  normal, and so $P_dP_d^* = P_d^*P_d$.  Hence for all $f,g\in\cH$,
  \begin{equation*}
    \ip{P_d^*f}{P_d^*g} = \ip{P_df}{P_dg}, \qquad \dinD.
  \end{equation*}
  By a standard result on inner products of nets of vectors,
  $\ip{P_d^* f}{P_d^* g}$ converges to $\ip{P^* f}{P^* g}$ and
  $\ip{P_d f}{P_d g}$ converges to $\ip{P f}{P g}$, in the limit.  So
  $\ip{P^*f}{P^*g} = \ip{Pf}{Pg}$.  Therefore $PP^*=P^*P$, and thus
  $P$ is normal.
\end{proof}

\section{Sums of pseudo-regular spaces}
\label{sec:sums-pseudo-regular}

The theorem below is a generalization of Theorem~4.3 of
\cite{Rovnyak2021}, and comes close to the form of
Theorem~\ref{maintheorem1} for regular spaces.  Since even finite
orthogonal direct sums of pseudo-regular spaces may not be
pseudo-regular (see Example~\ref{oct29a2024}), some rather strong
conditions are needed to guarantee that orthogonal direct sums of
pseudo-regular spaces are again pseudo-regular.

\begin{theorem}\label{ThC}
  Let $\{Q_\lambda\}_\linL$ be normal projections on a \krein space~$\cH$.
  For each
  $\linL$, let $\cM_\lambda, \widetilde\cM_\lambda$ be the ranges of
  $Q_\lambda, Q_\lambda^*$, and $\cN_\lambda, \widetilde\cN_\lambda$
  the kernels of $Q_\lambda, Q_\lambda^*$.  Suppose that for all
  $\lambda\neq\mu$,
  \begin{equation}\label{ThCa}
    \cM_\lambda\perp\cM_\mu,
    \quad
    \widetilde\cM_\lambda\perp\widetilde\cM_\mu,
    \quad
    \cM_\lambda\perp\widetilde\cM_\mu .
  \end{equation}
  The following are equivalent:
  \begin{enumerate}
  \item[(1)] For some and hence any norm $\|\cdot\|$ for~$\cH$, there
    is a constant $C>0$ such that
    $\|\sum_{\lambda\in \cF} Q_\lambda \| \le C$ for every finite
    subset $\cF$ of $\Lambda$.
    
    \smallskip
  \item[(2)] There is a normal projection $Q\in\lh$ such that for all
    $f$ in~$\cH$,
    \begin{equation}\label{ThCd}
      Qf=\sum^{\scalebox{0.5}{\text{MS}}}_{\linL} Q_\lambda f.      
    \end{equation}
  \end{enumerate}
  In this case, $Q^*f=\sum^{\scalebox{0.4}{\text{MS}}}_{\linL}
  Q_\lambda^* f$, and
  \begin{equation*}
    \begin{split}
      \ran Q &= \bigvee_\linL \cM_\lambda = \bigoplus_\linL
               \cM_\lambda,
      \quad \ker Q = \bigcap_\linL \cN_\lambda, \quad \text{and}\\
      \ran Q^* &= \bigvee_\linL \widetilde\cM_\lambda =
                 \bigoplus_\linL \widetilde\cM_\lambda ,
      \quad \ker Q^* = \bigcap_\linL \widetilde\cN_\lambda.
    \end{split}
  \end{equation*}
\end{theorem}

\begin{proof}
  Assume (1).  Let $D$ be the set of all finite subsets $\cF$ of
  $\Lambda$, directed by~$\supseteq$.  For each $\cF\in D$, set
  \begin{equation*}
    P_\cF=\sum_{\lambda\in\cF} Q_\lambda.
  \end{equation*}
  The relations \eqref{ThCa} have the equivalent forms
  $Q_\mu^*Q_\lambda= 0$, $Q_\mu Q_\lambda^*= 0$, and
  $Q_\mu Q_\lambda=0$.  By straightforward algebra, these equations
  imply that for each $\cF\in D$,
  \begin{equation*}
    P_\cF^*P_\cF = P_\cF P_\cF^*
    \quad\text{and}\quad
    P_\cF^2 = P_\cF,
  \end{equation*}
  that is, $P_\cF$ is a normal projection.  Further algebra shows that
  if $\cF$ and $\cF'$ are in $D$ and $\cF'\supseteq\cF$, then
  \begin{equation*}
    P_\cF P_{\cF'} = P_{\cF'} P_\cF = P_\cF.
  \end{equation*}
  By (1), $\|P_\cF\|\le C$, $\cF\in D$.  Thus the net of projections
  $\{P_\cF\}_{\cF\in D}$ satisfies the hypotheses of Theorems
  \ref{ThA} and \ref{ThB}.  By \eqref{ThAd} and \eqref{ThBa}, there is
  a normal projection $Q\in\lh$ such that for all $f$ in~$\cH$,
  \begin{equation*}
    \lim_{\cF\in D} P_\cF f = Qf, \qquad
    \lim_{\cF\in D} P_\cF^* f = Q^*f
  \end{equation*}
  in the strong topology of~$\cH$.  By the definition of Moore-Smith
  convergence, this is equivalent to \eqref{ThCd}, and so (2) follows.
  It is an elementary exercise to show that
  $Q^*f=\sum^{\scalebox{0.4}{\text{MS}}}_{\linL} Q_\lambda^* f$.
 
  Conversely, if (2) holds, using the argument in the proof of
  Theorem~\ref{thm:unctble_sum_op_ranges}, the operators
  $\{Q_\lambda \}_\linL$ have a uniform bound, that is, (1) holds.
  
  The final statement follows from Theorems \ref{ThA} and \ref{ThB}.
\end{proof}

As was pointed out, orthogonal subspaces of a \krein space need not
have zero intersection.  Nevertheless, it follows from the assumptions
of the theorem that for $\lambda \neq \mu$,
$\cM_\lambda \cap \cM_\mu = \{0\}$ since
$\cM_\lambda \subseteq \widetilde\cM_\mu^\bot$ implies that
$\cM_\lambda \cap \cM_\mu \subseteq \widetilde\cM_\mu^\bot \cap
\cM_\mu = \ker Q_\mu \cap \ran Q_\mu = \{0\}$.  The equality
$\widetilde\cM_\lambda \cap \widetilde\cM_\mu = \{0\}$ follows in the
same fashion.

Example~4.2 in \cite{Rovnyak2021} illustrates how normal projections
satisfying the conditions in the above theorem arise naturally as
wandering subspaces of isometries.

\section{Sums of quasi-pseudo-regular spaces}
\label{sec:sums-quasi-pseudo}

The goal now is to find a version of Theorem~\ref{maintheorem1} which
applies to qpr spaces.  The material from
Section~\ref{sec:sums-operator-ranges} plays a key role here.

\begin{theorem}\label{qpr_sums}
  Let $\{\cM_\lambda\}_{\linL}$ be a collection of pairwise orthogonal
  qpr spaces in a \krein space $\cH$.  Assume there are regular spaces
  $\{\cR_\lambda\}_\linL$ with corresponding selfadjoint projection
  operators $\{E_\lambda\}_\linL$, and operators $\{T_\lambda\}_\linL$
  in $\lh$ having neutral ranges $\{\cN_\lambda\}_\linL$, such that
  \begin{equation*}
    \cM_\lambda = \cR_\lambda \oplus \cN_\lambda, \qquad \linL.
  \end{equation*}
  Then the following are equivalent:
  \begin{enumerate}
  \item[(1)] For some and hence any norm for $\cH$ there is a $C >0$
    such that both $\|\sum_{\lambda\in\cF} E_{\lambda}\| \le C$ and
    $\|\sum_{\lambda\in\cF} T_{\lambda}\| \le C$ for every finite
    subset $\cF$ of $\Lambda$.
  \item[(2)] There is a selfadjoint projection $P \in \lh$ and an
    operator $T \in \cL(\cH^\Lambda, \cH)$ such that for all
    $f\in\cH$, and for all $x= (x_\lambda)_\lambda \in \cH^\Lambda$,  
    \begin{equation*}
      Pf = \sum^{\scalebox{0.5}{\text{MS}}}_{\linL} E_\lambda f,
      \qquad
      Tx = \sum^{\scalebox{0.5}{\text{MS}}}_{\linL} T_\lambda x_\lambda.
    \end{equation*}
    In this case, $\ran P = \bigoplusdot_\linL \cR_\lambda$, $\cN :=
    \ran T$ is a sum of $\{\cN_\lambda\}_\linL$, and
    \begin{equation*}
      \cM := \ran P \oplusdot \cN
    \end{equation*}
    is a sum of $\{\cM_\lambda\}_\linL$.
  \end{enumerate}
\end{theorem}
 
\begin{proof}
  The equivalence of (1) and (2) follows Theorems~\ref{maintheorem1}
  and~\ref{thm:unctble_sum_op_ranges} with
  $\ran P = \bigoplusdot_\linL \cR_\lambda$, $\cN := \ran T$ is a sum
  of $\{\cN_\lambda\}_\linL$.

  To see that $\cN$ is neutral,
  consider any $f,g\in\cN$.  Then by \eqref{aug20bb}, for some
  $x,y\in\cH^\Lambda$,
  \begin{equation*}
    f=Tx = \sum^{\scalebox{0.5}{\text{MS}}}_{\linL} T_\lambda x_\lambda, \qquad  
    g=Ty = \sum^{\scalebox{0.5}{\text{MS}}}_{\mu\in\Lambda} T_\mu y_{\mu} ,
  \end{equation*}
  and
  \begin{equation*}
    \ip{f}{g} = \ip{ \sum^{\scalebox{0.5}{\text{MS}}}_{\linL} T_\lambda
      x_\lambda}{ \sum^{\scalebox{0.5}{\text{MS}}}_{\mu\in\Lambda} T_\mu
      y_{\mu}} 
    = \sum_{\lambda,\mu\in\Lambda} \ip{T_\lambda x_\lambda}{T_\mu
      y_{\mu}} 
    = 0,
  \end{equation*}
  because the spaces $\{\ran T_\lambda\}_\linL$ are neutral and
  pairwise orthogonal.
 
  A similar argument shows that $\bigvee_\linL\cR_\lambda$ and $\cN$
  are orthogonal.  Thus
  $\cN\subseteq (\bigvee_\linL \cR_\lambda)^\perp$, so that
  $(\bigvee_\linL\cR_\lambda)\cap\cN = \{0\}$ because
  $\bigvee_\linL \cR_\lambda$ is regular.  Then by \cite[Theorem
  2.2]{FW1971}, the subspace
  \begin{equation*}
    \cM = \ran P \, \oplusdot \, \ran T 
  \end{equation*}
  is an operator range.  Finally, $\cM$ contains the finite sums of
  $\{\cM_\lambda\}_\linL$, and by construction,
  \begin{equation*}
    \cM \subseteq \sum^{\scalebox{0.5}{\text{MS}}}_{\linL} \cR_\lambda +
    \sum^{\scalebox{0.5}{\text{MS}}}_{\linL} \ran T_\lambda 
    \subseteq \sum^{\scalebox{0.5}{\text{MS}}}_{\linL} ( \cR_\lambda +
    \ran T_\lambda) 
    = \sum^{\scalebox{0.5}{\text{MS}}}_{\linL} \cM_\lambda .
  \end{equation*}
  Therefore $\cM$ is a sum of the $\{\cM_\lambda\}_\linL$.
\end{proof}

\begin{corollary}
  \label{cor:sums_of_countable_qpr_spaces}
  In Theorem \ref{qpr_sums}, when $\Lambda$ is countable the
  assumption that $\|\sum_{\lambda\in\cF} T_{\lambda}\| \le C$ for
  every finite subset $\cF$ of $\Lambda$ can be omitted.
\end{corollary}

\begin{proof}
  When $\Lambda$ is countable, in the proof of Theorem \ref{qpr_sums},
  scaling ensures that
  $\sum_{\lambda\in\cF} \|T_{\lambda}\| < \infty$.  Then Theorem
  \ref{thm:ctble_sum_op_ranges} is used to define the operator $T$ by
  \eqref{aug20b} and the proof is completed in the same way.
\end{proof}

\begin{corollary}
  \label{cor:appl_to_pr_spaces}
  In Theorem \ref{qpr_sums}, if the subspaces $\{\cM_\lambda\}_\linL$
  are pseudo-regular and
  $\sum^{\scalebox{0.4}{\text{MS}}}_{\linL} \ran T_\lambda$ is closed,
  then $\bigvee_\linL \cM_\lambda = \bigoplus_\linL \cM_\lambda$ is
  pseudo-regular and a unique sum of $\{\cM_\lambda\}_\linL$.
\end{corollary}

\begin{proof}
  By Theorem \ref{qpr_sums}, there is a sum $\cM$ of
  $\{\cM_\lambda\}_\linL$ of the form indicated in that theorem as qpr
  spaces.  By the assumption that
  $\sum^{\scalebox{0.4}{\text{MS}}}_{\linL} \ran T_\lambda$ is closed,
  $\cN$ is closed, and so $\cM$ is closed.  Hence by
  Lemma~\ref{lem:T_qpr_then_others_pr}, $\cM$ is pseudo-regular.  By
  Lemma~\ref{sum_closures}, the sum is unique.
\end{proof}

\begin{example}
  \label{exmpl:no_unique_sum}
  As already noted in Section~\ref{sec:sums-operator-ranges},
  orthogonal sums need not be unique.  As another example, consider
  qpr spaces $\{\cM_\lambda\}_{\linL}$ under the assumption that the
  regular parts are orthogonal and the neutral parts are all the same
  operator range $\cN$.  Assume also that the regular parts are
  summable.  Then the sum of these is a unique space by
  Theorem~\ref{maintheorem1}.  On the other hand, by
  Example~\ref{exmpl:unctble_sum_of_same_space} any operator range
  $\tilde{\cN}$ with
  $\cN \subseteq \tilde{\cN} \subseteq \overline{\cN}$, including
  $\tilde{\cN} = \cN$ and $\tilde{\cN} = \overline{\cN}$, are sums of
  the neutral spaces.  The space $\tilde{\cN}$ may or may not be
  closed, and so the sum of $\{\cM_\lambda\}_{\linL}$ can be
  pseudo-regular or just qpr, even when $\cN$ is not closed so that
  the spaces in $\{\cM_\lambda\}_{\linL}$ are qpr.
\end{example}

\bigskip
\noindent {\bf Funding:}  Alejandra Maestripieri was partially supported by the Air
   Force Office of Scientific Research (USA) grant FA9550-24-1-0433.  

\def\cprime{$'$} \def\cprime{$'$}
\providecommand{\bysame}{\leavevmode\hbox to3em{\hrulefill}\thinspace}
\providecommand{\MR}{\relax\ifhmode\unskip\space\fi MR }
\providecommand{\MRhref}[2]{%
  \href{http://www.ams.org/mathscinet-getitem?mr=#1}{#2}
}
\providecommand{\href}[2]{#2}


\end{document}